\PassOptionsToPackage{pdfpagelabels=false}{hyperref}
\documentclass[a4paper,oneside,11pt]{amsart}

\usepackage[utf8]{inputenc} 
\usepackage[T1]{fontenc} 

\usepackage[english=american]{csquotes} 

\usepackage[english]{babel} 

\usepackage{%
  amsmath,       
  amssymb,       
  graphicx,      
  mathtools,     
  microtype,     
  bm,            
  stmaryrd,      
  mathrsfs,      
}

\usepackage{nicematrix}

\usepackage{multicol}
\usepackage[shortlabels]{enumitem} 

\usepackage[citestyle=alphabetic,bibstyle=alphabetic,uniquename=init,autolang=hyphen,backend=biber,sortcites=true]{biblatex} 
\usepackage{amsthm}

\usepackage{varioref}

\usepackage{calc} 

\usepackage{tikz} 
\usetikzlibrary{calc,shapes.geometric,decorations.markings,decorations.pathmorphing,arrows,cd,quotes,babel}
\usepackage{pgfplots} 
\pgfplotsset{compat=1.17} 

\usepackage{hyperref} 
\usepackage{xurl}
\hypersetup{breaklinks=true}

\usepackage[nameinlink]{cleveref} 

\allowdisplaybreaks 

\newcommand*\N{\mathbb{N}}
\newcommand*\Z{\mathbb{Z}}
\newcommand*\Q{\mathbb{Q}}
\newcommand*\R{\mathbb{R}}

\newcommand*\F{\mathbb{F}}

\DeclarePairedDelimiterX\abs[1]\lvert\rvert{%
  \ifblank{#1}{\:\cdot\:}{#1}
}
\DeclarePairedDelimiterX\norm[1]\lVert\rVert{%
  \ifblank{#1}{\:\cdot\:}{#1}
}
\DeclarePairedDelimiterX{\inner}[2]{\langle}{\rangle}{%
  \ifblank{#1}{\:\cdot\:}{#1},\ifblank{#2}{\:\cdot\:}{#2}
}

\providecommand\given{}
\DeclarePairedDelimiterX\set[1]{\lbrace}{\rbrace}{
  \renewcommand*\given{\setSymbol[\delimsize]}
  #1
}

\DeclarePairedDelimiterX\free[1]{\langle}{\rangle}{
  \renewcommand\given{\nonscript\:\delimsize\vert\nonscript\:
    \mathopen{}}
  #1}

\DeclarePairedDelimiterX\open[2](){#1,#2}
\DeclarePairedDelimiterX\lopen[2](]{#1,#2}
\DeclarePairedDelimiterX\ropen[2][){#1,#2}
\DeclarePairedDelimiterX\closed[2][]{#1,#2}

\DeclareMathOperator\Span{span}

\DeclarePairedDelimiterX\gen[1]\langle\rangle{
  \ifblank{#1}{\:\cdot\:}{#1}
}

\DeclareMathOperator{\kernel}{ker} 
\DeclareMathOperator{\image}{im}
\DeclareMathOperator{\coker}{coker}
\DeclareMathOperator{\GL}{GL} 
\DeclareMathOperator{\SL}{SL} 
\DeclareMathOperator{\diag}{diag} 
\DeclareMathOperator{\Tr}{Tr} 
\DeclareMathOperator{\Hom}{Hom} 
\DeclareMathOperator{\Gal}{Gal} 
\DeclareMathOperator{\Lie}{Lie} 
\DeclareMathOperator{\Ht}{ht} 

\DeclareMathOperator{\gr}{gr}

\DeclareMathOperator{\red}{red}
\DeclareMathOperator{\sign}{sign}
\DeclareMathOperator{\Nrd}{Nrd}

\newcommand*{\lie}[1]{\mathfrak{#1}} 
\newcommand*\act{\,.\,} 
\newcommand*\edot{\:\cdot\:} 
\newcommand*\iso{\cong} 

\newcommand*\gs[1]{\mathcal{#1}}

\newcommand*\pmat[1]{\begingroup\linespread{1}\selectfont\begin{pNiceMatrix} #1 \end{pNiceMatrix}\endgroup}

\newcommand*\sO{\mathcal{O}}
\newcommand*\idm{\mathfrak{m}}

\newcommand*\basis{\mathcal{B}}

\makeatletter
\def\paragraph{\@startsection{paragraph}{4}%
  \z@\z@{-\fontdimen2\font}%
  {\normalfont\bfseries}}
\makeatother

\theoremstyle{plain}
\newtheorem{theorem}{Theorem}[section]
\newtheorem{conjecture}{Conjecture}
\newtheorem{proposition}[theorem]{Proposition}

\theoremstyle{remark}
\newtheorem{remark}[theorem]{Remark}

\title{On the Mod $p$ Cohomology of Pro-$p$ Iwahori Subgroups of $\SL_{n}(\Q_{p})$}
\date{\today}
\author{Daniel Kongsgaard}
\thanks{UC San Diego, \texttt{dkongsga@ucsd.edu}}
\thanks{I would like to thank Claus Sørensen for giving me interesting number theory research topics, and for his input and ideas during my research.}

\begin{document}

\begin{abstract}
  This paper can be seen as an update to part of the author's dissertation. We study the mod $p$ cohomology of the pro-$p$ Iwahori subgroups $I$ of $\SL_{n}(\Q_{p})$ (and $\GL_{n}(\Q_{p})$) for $n=2$ and $n=3$. Here we use the spectral sequence $E_{1}^{s,t} = H^{s,t}(\lie{g},\F_{p}) \Longrightarrow H^{s+t}(I,\F_{p})$ due to Sørensen, and we do explicit calculations with an ordered basis of $I$, which gives us a basis of $\lie{g} = \F_{p} \otimes_{\F_{p}[\pi]} \gr I$ that we use to calculate $H^{s,t}(\lie{g},\F_{p})$. We note that the multiplicative spectral sequence $E_{1}^{s,t} = H^{s,t}(\lie{g},\F_{p})$ collapses on the first page by noticing that all maps on each page are necessarily trivial, and this allows us to describe the above group cohomology groups and all cup products. Finally we note some connections to cohomology of central division algebras over $\Q_{p}$ and point out some future research directions.
\end{abstract}

\maketitle

\tableofcontents

\section{Introduction}%
\label{sec:intro}

The cohomology of Lie groups has a long history. In the late forties Chevalley and Eilenberg found that $H^{*}(G,\R) \iso H^{*}(\lie{g},\R)$ for a connected compact Lie group $G$ with Lie algebra $\lie{g}$ (cf.\ \cite{Chev}), and since then there has been much research into different types of Lie group cohomology.  In particular, the mod $p$ cohomology of a connected compact real Lie group has been well understood by Kac since the eighties (cf.\ \cite{Kac}), and the continuous mod $p$ cohomology $H^*(G,\F_p)$ of an equi-$p$-valued compact $p$-adic Lie group $G$ was already described by Lazard in the sixties (cf.\ \cite{Laz}). We note here that (except for Lazard's work) $H^{*}(G,\R)$ and $H^{*}(G,\F_{p})$ indicate the cohomology of $G$ as a topological space, and not continuous group cohomology, which can be thought of as the cohomology of the classifying space $BG$.

This paper's main interest is the continuous mod $p$ cohomology $H^{*}(G,\F_{p})$ of compact $p$-adic Lie groups $G$ for some specific $G$. Since $p$-adic Lie groups are totally disconnected, working with them requires very different methods than what Chevalley and Eilenberg or Kac used for real Lie groups, and we have to follow the ideas of Lazard (see \cite{Laz}) and Serre. In particular we need a $p$-valuation on $G$ (and on the completed group algebras associated with $G$), and we work with the graded \emph{Lazard} Lie algebra $\lie{g} = \F_{p} \otimes_{\F_{p}[\pi]} \gr G$ attached to $G$. We will repeatedly use that Sørensen (in \cite{Sor}) showed that $H^{*}(\lie{g},\F_{p})$ determines $H^{*}(G,\F_{p})$ via a multiplicative spectral sequence \[ E_{1}^{s,t} = H^{s,t}(\lie{g},\F_{p}) \Longrightarrow H^{s+t}(G,\F_{p}). \] When $G$ is equi-$p$-valuable, we get that $\lie{g}$ is concentrated in a single degree, and Lazard showed that $H^{*}(G,\F_{p}) \iso \bigwedge H^{1}(\lie{g},\F_{p})$, while Sørensen showed that this also follows from the above spectral sequence. We are interested in cases where $G$ is \emph{not} equi-$p$-valuable, and we note that the spectral sequence of Sørensen allows us to work purely with $G$ and $\lie{g}$ without having to worry about the completed group algebras $\Lambda(G) = \Z_{p}\llbracket G \rrbracket$ and $\Omega(G) = \F_{p}\llbracket G \rrbracket$ (though these do show up in Sørensen's arguments).

Before describing our particular results in the following paragraph, we emphasize the following remark of Sørensen from \cite{Sor}: It is known (due to Lazard) that any compact $p$-adic Lie group contains an open equi-$p$-valuable subgroup (see \cite[Chap.~V~2.2.7.1]{Laz}), which gives the impression that the distinction between $p$-valued and equi-$p$-valued groups is somewhat nuanced, which is true for some questions. But there are many examples of naturally occurring $p$-valuable groups $G$ which are not equi-$p$-valuable, where detailed information about $H^{*}(G,\F_{p})$ is important. For example unipotent groups (i.e., the $\Z_{p}$-points of the unipotent radical of a Borel in a split reductive group), Serre's standard groups with $e>1$ as in \cite[Lem.~2.2.2]{Laz-complements}, pro-$p$ Iwahori subgroups for large enough $p$, and $1 + \idm_{D}$ where $D$ is the quaternion division algebra over $\Q_{p}$ for $p > 3$ (or more generally a central division algebra over $\Q_{p}$). Sørensen explicitly calculates $H^{*}\bigl( (1+\idm_{D})^{\Nrd = 1} , \F_{p} \bigr)$ for $p>3$ and uses it to describe $H^{*}(1+\idm_{D},\F_{p})$, and he notes that $1+\idm_{D}$ plays an important role both in number theory (in the Jacquet-Langlands correspondence for instance, see \cite{JL}) and algebraic topology, where $1+\idm_{D}$ is known as the (strict) Morava stabilizer in stable homotopy theory, and $H^{*}(1+\idm_{D},\F_{p})$ somehow controls certain localization functors with respect to Morava $K$-theory (see e.g.\ \cite{Henn}).

In this paper we focus on the case of pro-$p$ Iwahori subgroups of $\SL_{n}$ (and to a lesser extend $\GL_{n}$) over $\Q_{p}$ for $n=2$ or $n=3$, and our work will build on ideas of Lazard and Serre from their more general (but not yet finished) description of the case when $G$ is not equi-$p$-valued. We especially rely on the refinement of these ideas as described by Sørensen and Schneider in \cite{Sor} and \cite{Sch-notes}. (We note that the cases of $n=4$ and the case over quadratic extensions $F/\Q_{p}$ for $n=2$ are also investigated in \cite{thesis}.) We explicitly calculate the algebra structure of $H^{*}(I,\F_{p})$ for the pro-$p$ Iwahori subgroups $I_{\SL_{2}(\Q_{p})} \subseteq \SL_{2}(\Z_{p})$ and $I_{\SL_{3}(\Q_{p})} \subseteq \SL_{3}(\Z_{p})$, and we note that the first of these is isomorphic as algebras to $H^{*}\bigl( (1+\idm_{D})^{\Nrd = 1},\F_{p} \bigr)$ while the second is at least isomorphic as graded vector spaces (and we conjecture also as algebras). Here $D/\Q_{p}$ is the central division algebra of dimension $n^{2}$ for $n=2$ and $n=3$, respectively. We finish the paper by mentioning some future research directions and a conjecture on the connection between the mod $p$ cohomology of $(1+\idm_{D})^{\Nrd = 1}$ (resp.\ $1+\idm_{D}$) for central division algebras and $I_{\SL_{n}(\Q_{p})}$ (resp.\ $I_{\GL_{n}(\Q_{p})}$).

\subsection{Background and layout of the paper}%
\label{subsec:background-iwa}

This work can be seen as a continuation of the recent work on the mod $p$ cohomology of pro-$p$ Iwahori subgroups. E.g.\ the work by Schneider and Olivier (see \cite{SchOll-modular,SchOll-torsion,Sch-smooth}) working with pro-$p$ Iwahori-Hecke modules and the work by Koziol (see \cite{Koziol}) computing $H^{1}(I,\pi)$ as a $\mathcal{H}$-algebra (where $\mathcal{H}$ is the pro-$p$ Iwahori-Hecke algebra and $\pi$ is a mod $p$ principal series representation of $\GL_{n}(F)$ for some $p$-adic field $F$). Work by Cornut and Ray (cf.\ \cite{Generators}) finding a minimal set of topological generators of the pro-$p$ Iwahori subgroup of a split reductive group over $\Z_{p}$ is also relevant, since the number of generators can be used to find the cohomological dimension of $H^{1}(I,\F_{p})$. Overall all of this work can be seen as part of the search for a mod $p$ and $p$-adic local Langlands program.

We start the paper by introducing the setup and notation in the next subsection, after which we briefly note that $H^{*}(I_{\GL_{n}(\Q_{p})}, \F_{p})$ can be fully described if we know $H^{*}(I_{\SL_{n}(\Q_{p})}, \F_{p})$. We give an overview of the techniques we use throughout the paper in the next section. We then explicitly calculate the algebra structure of $H^{*}(I,\F_{p})$ for the pro-$p$ Iwahori subgroups $I = I_{\SL_{2}(\Q_{p})} \subseteq \SL_{2}(\Z_{p})$ and (using computer calculations) $I = I_{\SL_{3}(\Q_{p})} \subseteq \SL_{3}(\Z_{p})$, and we note that the first of these is isomorphic as algebras to $H^{*}\bigl( (1+\idm_{D})^{\Nrd = 1},\F_{p} \bigr)$ for $n=2$ and the second at least as graded vector spaces for $n=3$. We note that these calculations heavily rely on results of Sørensen and Fuks (see \cite{Sor} and \cite{Fuks}). In particular the multiplicative spectral sequence \[ E_{1}^{s,t} = H^{s,t}(\lie{g},\F_{p}) \Longrightarrow H^{s+t}(I,\F_{p}) \] of Sørensen (see \cite{Sor}) collapses on the first page in both cases, which allows us to calculate the continuous group cohomology via the Lie algebra cohomology of $\lie{g}$. Here $\lie{g} = \F_{p} \otimes_{\F_{p}[\pi]} \gr I$ is the graded Lazard Lie algebra associated with $I$.

We note that this paper in large parts are repeating material from the authors dissertation (see \cite{thesis}), but some parts have been refined since then. In particular we give a small correction (not to a calculation, but to a comment) and add new calculations in the $\SL_{3}$-case, which fully describes the cup proudcts of the cohomology now, where before only the dimensions where calculated.

\subsection{Setup and notation}%
\label{subsec:setup-iwa}

Let $p$ be an odd prime (further restricted later).

\paragraph{Field extension of \texorpdfstring{$\Q_{p}$}{Qp}.} We fix a finite extension of $F/\Q_{p}$ of degree $\ell$ with valuation ring $\sO_{F}$ and maximal ideal $\idm_{F} = (\varpi_{F}) \subseteq \sO_{F}$. Let $e = e(F/\Q_{p})$ be the ramification index and $f = f(F/\Q_{p})$ the inertia degre eof the extension $F/\Q_{p}$. Let furthermore $v$ be the valuation on $F$ for which $v(p) = 1$, and thus $v(\varpi_{F}) = \frac{1}{e}$.

\paragraph{\texorpdfstring{$\exp$}{exp} and \texorpdfstring{$\log$}{log}.} Given a finite field extension $F/\Q_{p}$ with valuation ring $\sO_{F}$ and maximal ideal $\idm_{F}$ with $p\sO_{F} = \idm_{F}^{e}$, we get by \cite[Prop.~(5.5)]{Neukirch} (noting that we will ensure that $1 > \frac{e}{p-1}$ later) that the power series
\begin{equation*}
  \exp(x) = 1 + x + \frac{x^{2}}{2!} + \frac{x^{3}}{3!} + \dotsb \quad \text{ and } \quad \log(1+z) = z - \frac{z^{2}}{2} + \frac{z^{3}}{3} - \dotsb,
\end{equation*}
are two mutually inverse isomorphisms (and homeomorphisms)
\[
  \begin{tikzcd}
    \idm_{F} \ar[r, yshift=0.7ex, "\exp"] & U_{F}^{(1)}. \ar[l, yshift=-0.7ex, "\log"]
  \end{tikzcd}
\]
Note that this implies that a $\Z_{p}$-basis of $\idm_{F}$ translates to a $\Z_{p}$-basis of $U_{F}^{(1)} = 1+\idm_{F}$ via $\exp$.

\paragraph{Big-\texorpdfstring{$O$}{O} notation.} For elements of $\sO_{F}$ we write $x = y + O(p^{r})$ if and only if $x-y \in p^{r}\sO_{F}$.

\paragraph{Matrices.} Let $E_{ij}$ denote the matrix with $1$ in the $(i,j)$ entry, and zeroes in all other entries, and write $1_{n}$ for the identity matrix in $M_{n}(F)$. Let $A = (a_{ij})$. We write $A = \diag(a_{1},\dotsc,a_{n})$ for the diagonal matrix in $M_{n}(F)$ with entries $a_{ii}=a_{i}$ in the diagonal, and $A = \diag_{i_{1},\dotsc,i_{k}}(a_{1},\dotsc,a_{k})$ for the diagonal matrix in $M_{n}(F)$ with entries $a_{i_{\ell}i_{\ell}} = a_{\ell}$ for $\ell = 1,\dotsc,k$, ones in the other diagonal entries and zeroes in all other entries. Finally, we write $A^{\top}$ for the transpose matrix of $A$.

\paragraph{Dual basis.} Let $V$ be a $\F_{p}$-vector space with basis $\basis = (e_{1},\dotsc,e_{d})$. Then we let $\basis^{*} = (e_{1}^{*},\dotsc,e_{d}^{*})$ be the dual basis of $\Hom_{\F_{p}}(V,\F_{p})$ defined by $e_{i}^{*}(e_{i}) = \delta_{ij}$, where $\delta_{ij}$ is the Kronecker delta function. Now consider two vector spaces $V$ and $W$ with bases $\basis_{V}$ and $\basis_{W}$. Given a linear map $d \colon V \to W$ with matrix $A$ when described in these bases, it is a well known fact from linear algebra that the dual map $d^{*} \colon \Hom_{\F_{p}}(W,\F_{p}) \to \Hom_{\F_{p}}(V,\F_{p})$ has matrix $A^{\top}$ when described in the dual bases $\basis_{V}^{*}$ and $\basis_{W}^{*}$. We will often use this without mention and abuse notation writing $d$ and $d^{\top}$ for these matrices.

\paragraph{Lazard theory.} For an introduction to Lazard theory see \cite{Sch}. We will let $\lie{g} = \F_{p} \otimes_{\F_p[\pi]} \gr I$ be the Lazard Lie algebra corresponding to the pro-$p$ Iwahori subgroup $I$. Furthermore, recall that a sequence of elements $(g_1,\dotsc,g_r)$ in $G$ is called an \emph{ordered basis} of $(G,\omega)$ if the map $\Z_{p}^{r} \to G$ given by $(x_{1},\dotsc,x_{r}) \mapsto g_{1}^{x_{1}} \dotsb g_{r}^{x_{r}}$ is a bijection (and hence, by compactness, a homeomorphism) and
\begin{equation*}
  \omega(g_1^{x_1}\dotsb g_r^{x_r}) = \min_{1 \leq i \leq r}(\omega(g_i)+v(x_i)) \qquad \text{for any } x_1,\dotsc,x_r\in\Z_p.
\end{equation*}

\paragraph{Algebraic groups.} We work with schemes using the functorial approach and notation described in \cite{Jan}, but we generally brush over the details. For more in depth introduction to these concepts, we refer to \cite{Con-book} and \cite{Jan}. Note that we only think of more general algebraic groups for this introductory section, and we will immediately switch to concrete groups like $G = \SL_{n}(\Q_{p})$ in the next section.

\paragraph{Fixed groups and roots.} We fix a split and connected reductive algebraic $F$-group $\gs{G}$, and consider the locally profinite group $G = \gs{G}(F)$. We then fix split maximal torus $\gs{T} \subseteq \gs{G}$ and let $T = \gs{T}(F)$. In $T$ we have a maximal compact subgroup $T^{0}$ and its Sylow pro-$p$ subgroup $T^{1}$.

Let $\Phi = \Phi(\gs{G},\gs{T})$ be the root system of $\gs{G}$ with respect to $\gs{T}$, and let $(X^{*}(T),\Phi,X_{*}(T),\Phi^{\vee})$ be the associated root datum. Fix a system of positive roots $\Phi^{+}$ and let $\Delta \subseteq \Phi^{+}$ be the simple roots. For any $\alpha \in \Phi$ we have the root subgroup $\gs{U}_{\alpha} \subseteq \gs{G}$ with Lie algebra $\Lie \gs{U}_{\alpha} =  (\Lie \gs{G})_{\alpha}$. We let $U_{\alpha} = \gs{U}_{\alpha}(F)$ and choose an isomorphism $x_{\alpha} \colon F \xrightarrow{\iso} U_{\alpha}$ such that $tx_{\alpha}(x)t^{-1} = x_{\alpha}(\alpha(t)x)$ for $t \in T$ and $x \in F$. For $r \in \Z_{\geq 0}$ we let $U_{\alpha,r} = x_{\alpha}(\idm_{F}^{r})$.

\paragraph{Coxeter number and \texorpdfstring{$p$}{p}.} Let $h$ be the Coxeter number of $\gs{G}$ and assume from now on that $p-1 > eh$. We note that this translates to $p > n+1$ for $\gs{G} = \SL_{n}$.

\paragraph{Pro-\texorpdfstring{$p$}{p} Iwahori subgroups.} We follow the definitions of \cite{SchOll-modular} with $\gs{G}, \gs{T}$ and $\gs{U}_{\alpha}$ as above. Let $I$ be the pro-$p$ Iwahori subgroup of $G$ (associated with a positive chamber as in \cite{SchOll-modular}, but we do not need the exact definition). We note by \cite[Lem.~2.1(i)]{SchOll-modular} and the proof of \cite[Lem.~2.3]{SchOll-modular} that $I$ has the following factorization: Multiplication defines a homeomorphism
\begin{equation}\label{eq:Iwahori-factor}
  \prod_{\alpha \in \Phi^{-}} U_{\alpha,1} \times T^{1} \times \prod_{\alpha \in \Phi^{+}} U_{\alpha,0} \xrightarrow{\iso} I,
\end{equation}
where the products are ordered in an arbitrarily chosen way. For a more detailed introduction to these pro-$p$ groups we refer to \cite{SchOll-modular}.

\paragraph{Pro-\texorpdfstring{$p$}{p} Iwahori subgroups of \texorpdfstring{$\GL_{n}(F)$}{GLn(F)} and \texorpdfstring{$\SL_{n}(F)$}{SLn(F)}.} In this paper, we will only work with pro-$p$ Iwahori subgroups of $\SL_{n}(F)$ (or $\GL_{n}(F)$), which simplifies the definitions. When $\gs{G} = \SL_{n}$ (or $\gs{G} = \GL_{n}$), we can always take $\gs{T}$ the diagonal maximal torus, and we can take $I$ to be the subgroup of $\gs{G}(\sO_{F})$ which is upper triangular and unipotent modulo $\varpi_{F}$. In this case we have that $U_{\alpha,1}$ for $\alpha \in \Phi^{-}$ correspond to entries below the diagonal and $U_{\alpha,0}$ for $\alpha \in \Phi^{+}$ corresponds to the entries above the diagonal.

\paragraph{\texorpdfstring{$p$}{p}-valuation on \texorpdfstring{$I$}{I}.} By a recent preprint by Lahiri and Sørensen (cf.\ \cite[Prop.~3.4]{IwaBasis}), we know (since $p-1 > eh$) that $I$ admits a $p$-valuation $\omega$ satisfying the properties:
\begin{enumerate}[(a)]
  \item $\omega$ is compatible with Iwahori factorization \eqref{eq:Iwahori-factor} of $I$ (cf.\ \cite[Def.~3.3]{IwaBasis}).
  \item $\omega(x_{\alpha}(x)) = v(x) + \frac{\Ht(\alpha)}{eh}$ where $\begin{dcases}
    x \in \idm_{F} & \text{if } \alpha \in \Phi^{-}, \\
    x \in \sO_{F} & \text{if } \alpha \in \Phi^{+}.
  \end{dcases}$
  \item $\omega(t) = \frac{1}{e} \cdot \sup\set{n \in \N : t \in T^{n}}$ for $t \in T^{1}$.
\end{enumerate}

\paragraph{Ordered basis of \texorpdfstring{$I$}{I}.} Let $\set{b_{1},\dotsc,b_{\ell}}$ be a $\Z_{p}$-basis of $\sO_{F}$, 
where $\ell = [F:\Q_{p}]$. Then
$\bigl( x_{\alpha}(b_{1}), \dotsc, x_{\alpha}(b_{\ell}) \bigr)$ is an ordered basis for $U_{\alpha,0}$ when $\alpha \in \Phi^{+}$, and $\bigl( x_{\alpha}(\varpi_{F}b_{1}), \dotsc, x_{\alpha}(\varpi_{F}b_{\ell}) \bigr)$ is an ordered basis for $U_{\alpha,1}$ when $\alpha \in \Phi^{-}$. Furthermore, when $G$ is semisimple and simply connected, we have that the simple coroots $\set{ \alpha^{\vee} : \alpha \in \Delta }$ form a $\Z$-basis of $X_{*}(T)$, and thus $\bigl( \alpha^{\vee}(\exp(\varpi_{F}b_{1})), \dotsc, \alpha^{\vee}(\exp(\varpi_{F}b_{\ell})) \bigr)_{\alpha \in \Delta}$ form an ordered basis of $T^{1}$. By \cite[Cor.~3.6]{IwaBasis}, given orderings of $\Phi^{+}$ and $\Phi^{-}$, and assuming that $G$ is semisimple and simply connected, we now get: the sequence of elements
\begin{enumerate}[$\bullet$]
  \item $\bigl( x_{\alpha}(\varpi_{F}b_{1}), \dotsc, x_{\alpha}(\varpi_{F}b_{\ell}) \bigr)_{\alpha \in \Phi^{-}}$,
  \item $\bigl( \alpha^{\vee}(\exp(\varpi_{F}b_{1})), \dotsc, \alpha^{\vee}(\exp(\varpi_{F}b_{\ell})) \bigr)_{\alpha \in \Delta}$,
  \item $\bigl( x_{\alpha}(b_{1}), \dotsc, x_{\alpha}(b_{\ell}) \bigr)_{\alpha \in \Phi^{+}}$
\end{enumerate}
forms an ordered basis of $(I,\omega)$ (with $\omega$ from the previous paragraph) which is a saturated $p$-valued group. Here, \cite{IwaBasis} notes that the $p$-valuation from the previous paragraph on this basis is given by (cf.\ \cite[Prop.~3.4]{IwaBasis})
\begin{equation}
  \label{eq:Iwa-p-val-basis}
  \begin{dcases}
    \omega\bigl( x_{\alpha}(\varpi_{F}b_{i}) \bigr) = \frac{1}{e} + \frac{\Ht(\alpha)}{eh}  & \alpha \in \Phi^{-} \\
    \omega\bigl( \alpha^{\vee}(\exp(\varpi_{F}b_{i})) \bigr) = \frac{1}{e} & \alpha \in \Delta \\
    \omega\bigl( x_{\alpha}(b_{i}) \bigr) = \frac{\Ht(\alpha)}{eh} & \alpha \in \Phi^{+}.
  \end{dcases}
\end{equation}

We note that the above argument uses that $\exp \colon \idm_{F} = (\varpi_{F}) \to U_{F}^{(1)} = 1 + \idm_{F}$ takes a basis to a basis and that $\set{ \varpi_{F}b_{1}, \dotsc, \varpi_{F}b_{\ell} }$ is a $\Z_{p}$-basis of $\idm_{F} = \varpi_{F}\sO_{F}$. 

When $\gs{G} = \SL_{n}$, we have that $\Phi = \set{ \varepsilon_{i}-\varepsilon_{j} \given 1 \leq i,j \leq n, i\neq j }$ and can take
\begin{equation*}
  \Delta = \set{\alpha_{1}=\varepsilon_{1}-\varepsilon_{2}, \alpha_{2}=\varepsilon_{2}-\varepsilon_{3}, \dotsc, \alpha_{n-1}=\varepsilon_{n-1}-\varepsilon_{n}},
\end{equation*}
where $\varepsilon_{i}$ is the map that takes a diagonal matrix to its $i$-th diagonal entry. In this case $\alpha_{i}^{\vee}(u) = \diag(1,\dotsc,1,u,u^{-1},1,\dotsc,1) = \diag_{i,i+1}(u,u^{-1})$, where the non-trivial entries are the $i$-th and $(i+1)$-th entries. Since the second non-trivial entry of these matrices are always just the inverse of the first entry, we will abuse notation and write $\diag_{i,i+1}(u) = \diag_{i,i+1}(u,u^{-1})$. This together with the above gives us the following ordered basis (in the listed order and with a chosen ordering of $\set{ (i,j) : 1 \leq i,j \leq n }$) in the case $\gs{G} = \SL_{n}$:
\begin{enumerate}[$\bullet$]
  \item $\bigl( 1_{n}+\varpi_{F}b_{1}E_{ij}, \dotsc, 1_{n}+\varpi_{F}b_{\ell}E_{ij} \bigr)_{1 \leq j < i \leq n}$,
  \item $\bigl( \diag_{i,i+1}(\exp(\varpi_{F}b_{1})), \dotsc, \diag_{i,i+1}(\exp(\varpi_{F}b_{\ell})) \bigr)_{i=1,\dotsc,n-1}$,
  \item $\bigl( 1_{n}+b_{1}E_{ij}, \dotsc, 1_{n}+b_{\ell}E_{ij} \bigr)_{1 \leq i < j \leq n}$.
\end{enumerate}
Here the $p$-valuation described in \eqref{eq:Iwa-p-val-basis} is given by
\begin{equation}\label{eq:Iwa-p-val-basis-SLn}
  \begin{dcases}
    \omega\bigl( 1_{n} + \varpi_{F}b_{m}E_{ij} \bigr) = \frac{1}{e} + \frac{j-i}{eh} & j < i, \\
    \omega\bigl( \diag_{i,i+1}(\exp(\varpi_{F}b_{m})) \bigr) = \frac{1}{e} & i = 1,\dotsc,n-1, \\
    \omega\bigl( 1_{n}+b_{m}E_{ij} \bigr) = \frac{j-i}{eh} & i < j
  \end{dcases}
\end{equation}
on the above ordered basis.

Finally note that an ordered basis of $\GL_{n}$ can be obtained from an ordered basis of $\SL_{n}$ by adding non-trivial elements of the center, which in the above corresponds to adding $\bigl(  \exp(\varpi_{F}b_{1})1_{n}, \dotsc, \exp(\varpi_{F}b_{\ell})1_{n} \bigr)$ to the middle item above (adding the root $\varepsilon_{1} + \dotsb + \varepsilon_{n}$), and the $p$-valuation on these is still $\frac{1}{e}$. In particular, $I_{\GL_{n}(\Q_{p})} = (1+p\Z_{p}) \times I_{\SL_{n}(\Q_{p})}$ (which will be discussed in more detail later).

\paragraph{Cohomology.} We denote (using the Chevalley-Eilenberg complex) the mod $p$ Lie algebra cohomology of any $\F_{p}$-Lie algebra $\lie{g}$ by $H^{\bullet}(\lie{g}, \F_{p})$, while we write $H^{\bullet}(G,\F_{p})$ for the mod $p$ continuous group cohomology of a topological group $G$. We introduce filtrations and then gradings on the cohomology and use the notation $H^{s,t} = \gr^{s}H^{s+t}$ for any type of cohomology $H$.

\paragraph{Spectral sequences.} A cohomological spectral sequence in this paper is a choice of $r_0 \in \N$ and a collection of
\begin{enumerate}[$\bullet$]
  \item $\F_{p}$-modules $E_r^{s,t}$ for each $s,t \in \Z$ and all integers $r \geq r_0$
  \item differentials $d_r^{s,t} \colon E_r^{s,t} \to E_r^{s+r,t+1-r}$ such that $d_r^2 = 0$ and $E_{r+1}$ is isomorphic to the homology of $(E_r,d_r)$, i.e.,
  \[
    E_{r+1}^{s,t} = \frac{\kernel(d_r^{s,t} \colon E_r^{s,t} \to E_r^{s+r,t+1-r})}{\image(d_r^{s-r,t+r-1} \colon E_r^{s-r,t+r-1} \to E_r^{s,t})}.
  \]
\end{enumerate}
For a given $r$, the collection $(E_r^{s,t},d_r^{s,t})_{s,t\in\Z}$ is called the $r$-th page. A spectral sequence \emph{converges} if $d_r$ vanishes on $E_r^{s,t}$ for any $s,t$ when $r\gg0$. In this case $E_r^{s,t}$ is independent of $r$ for sufficiently large $r$, we denote it by $E_{\infty}^{s,t}$ and write
  \[
    E_{r}^{s,t} \Longrightarrow E_\infty^{s+t}.
  \]
  Also, we say that the spectral sequence collapses at the $r'$-th page if $E_{r} = E_{\infty}$ for all $r \geq r'$, but not for $r < r'$. Finally, when we have terms $E_\infty^{n}$  with a natural filtration $F^\bullet E_\infty^n$ (but no natural double grading), we set $E_\infty^{s,t} = \gr^{s} E_\infty^{s+t}= F^{s}E_\infty^{s+t}/F^{s+1}E_\infty^{s+t}$.


\subsection{A note on the \texorpdfstring{$\GL_{n}$}{GLn}-case}%
\label{subsec:note-on-GLn}

Before continuing we will give a short note on this paper's connection with the mod $p$ continuous group cohomology of the pro-$p$ Iwahori subgroup of $\GL_{n}(\Q_{p})$. We do this by comparing with a similar result for subgroups of central division algebras, so let's briefly recall the setup from \cite{Sor}.

Let $D$ be the central division algebra over $\Q_{p}$ of dimension $n^{2}$ and invariant $\frac{1}{n}$. Recall the following setup from \cite[Sect.~6.3]{Sor}: The valuation $v_{p}$ on $\Q_{p}$ extends uniquely to a valuation $\tilde{v} \colon D^{\times} \to \frac{1}{n}\Z$ by the formula $\tilde{v}(x) = \frac{1}{n}v\bigl(\Nrd_{D/\Q_{p}}(x)\bigr)$, and the valuation ring $\sO_{D} = \set{ x : \tilde{v}(x) >0 }$ is the maximal compact subring of $D$. It is local with maximal ideal $\idm_{D} = \set{ x : \tilde{v}(x) > 0 }$ and residue field $\F_{D} \iso \F_{p^{n}}$. Furthermore, we can pick $\varpi_{D}$ such that $\tilde{v}(\varpi_{D}) = \frac{1}{n}$, $\idm_{D} = \varpi_{D}\sO_{D} = \sO_{D}\varpi_{D}$ and $p = \varpi_{D}^{n}$.

\begin{proposition}\label{prop:GLn-connection}
  Assume that $p>n+1$ and let $I_{\SL_{n}(\Q_{p})}$ and $I_{\GL_{n}(\Q_{p})}$ be the pro-$p$ Iwahori subgroup of $\SL_{n}(\Q_{p})$ and $\GL_{n}(\Q_{p})$, respectively. Then
  \begin{equation*}
    H^{*}\bigl( I_{\GL_{n}(\Q_{p})}, \F_{p} \bigr) \iso H^{*}\bigl( I_{\SL_{n}(\Q_{p})}, \F_{p} \bigr) \otimes \F_{p}[\varepsilon],
  \end{equation*}
  where $\F_{p}[\varepsilon]$ denotes the dual numbers ($\varepsilon^{2} = 0$). Similarly
  \begin{equation*}
    H^{*}\bigl( 1+\idm_{D} , \F_{p} \bigr) \iso H^{*}\bigl( (1+\idm_{D})^{\Nrd=1} , \F_{p} \bigr) \otimes \F_{p}[\varepsilon].
  \end{equation*}
\end{proposition}
\begin{proof}
  We note that \cite[Sect.~6.3]{Sor} proves the second part of the proposition as follows.

  Consider the short exact sequence
  \[
    \begin{tikzcd}
      1 \ar[r] & (1+\idm_{D})^{\Nrd=1} \ar[r] & 1+\idm_{D} \ar[r,"\Nrd"] & 1+p\Z_{p} \ar[r] & 1,
    \end{tikzcd}
  \]
  and note that $p>n+1$ implies that $\frac{1}{n} \in \Z_{p}$, so $(\edot)^{\frac{1}{n}}$ makes sense on $1+p\Z_{p}$ viewed as a central subgroup of $1+\idm_{D}$. Thus we can write $1+\idm_{D} \iso (1+\idm_{D})^{\Nrd=1} \times (1+p\Z_{p})$, and by Künneth
  \begin{align*}
    H^{*}\bigl( 1+\idm_{D},\F_{p} \bigr) &\iso H^{*}\bigl( (1+\idm_{D})^{\Nrd=1},\F_{p} \bigr) \otimes H^{*}\bigl( 1+p\Z_{p},\F_{p} \bigr) \\
    &\iso H^{*}\bigl( (1+\idm_{D})^{\Nrd=1},\F_{p} \bigr) \otimes \F_{p}[\varepsilon].
  \end{align*}
  Here the identification $H^{*}(1+p\Z_{p},\F_{p}) \iso \F_{p}[\varepsilon]$ can be viewed as a very special case of Lazard's isomorphism.

  A similar argument applies to the first part of the proposition:

  Consider the short exact sequence
  \[
    \begin{tikzcd}
      1 \ar[r] & I_{\SL_{n}(\Q_{p})} \ar[r] & I_{\GL_{n}(\Q_{p})} \ar[r,"\det"] & 1+p\Z_{p} \ar[r] & 1,
    \end{tikzcd}
  \]
  and note that $p>n+1$ implies that $\frac{1}{n} \in \Z_{p}$, so $(\edot)^{\frac{1}{n}}$ makes sense on $1+p\Z_{p}$ viewed as a central subgroup of $I_{\GL_{n}(\Q_{p})}$. Thus we can write $I_{\GL_{n}(\Q_{p})} \iso I_{\SL_{n}(\Q_{p})} \times (1+p\Z_{p})$, and by Künneth
  \begin{align*}
    H^{*}\bigl( I_{\GL_{n}(\Q_{p})},\F_{p} \bigr) &\iso H^{*}\bigl( I_{\SL_{n}(\Q_{p})},\F_{p} \bigr) \otimes H^{*}\bigl( 1+p\Z_{p},\F_{p} \bigr) \\
    &\iso H^{*}\bigl( I_{\SL_{n}(\Q_{p})},\F_{p} \bigr) \otimes \F_{p}[\varepsilon].
  \end{align*}
  We note that here $1+p\Z_{p}$ is the center of $I_{\GL_{n}(\Q_{p})}$ since $I_{\GL_{n}(\Q_{p})} \subseteq \GL_{n}(\Z_{p})$ and $\Z_{p}^{\times} = \mu_{p-1} \times (1+p\Z_{p})$, so when discussing pro-$p$ groups we can ignore the $\mu_{p-1}$ term.
\end{proof}


\section{Techniques}%
\label{sec:tech-iwa}

In this section we will describe how to calculate information about the cohomology of a $p$-valuable group by using its Lazard Lie algebra.

Let $(G, \omega)$ be a $p$-valuable group and let $k$ be a perfect field of characteristic $p$ (later we will assume that $k = \F_{p}$). In this section we will describe how the spectral sequence
\begin{equation}\label{eq:spec-sec-tech}
  E_{1}^{s,t} = H^{s,t}(\lie{g}, k) \Longrightarrow H^{s+t}(G, k)
\end{equation}
from \cite[§6.1]{Sor} can be used to calculate information about the dimensions of $H^{n}(G,k)$ for varying $n$ and information about the cup product on $H^{*}(G,k)$. After this, we will then discuss how this applies to pro-$p$ Iwahori subgroups $I$ of $\SL_{n}(\Q_{p})$ for $n=2$ and $n=3$.

Recall that $\lie{g}$ in the above spectral sequence is given by $\lie{g} = k \otimes_{\F_{p}[\pi]} \gr G$, so to describe $\lie{g}$, we first need a good description of the $\F_{p}[\pi]$-Lie algebra $\gr G$. To get this description, suppose that we have an ordered basis $(g_{1},\dotsc,g_{d})$ of $G$, so that $\omega(g) = \min_{i = 1,\dotsc,d} \bigl( \omega(g_{i}) + v_{p}(x_{i}) \bigr)$ for $g = g_{1}^{x_{1}} \dotsb g_{d}^{x_{d}}$, and recall that $\bigl( \sigma(g_{1}),\dotsc,\sigma(g_{d}) \bigr)$ is a basis of $\gr G$, where $\sigma(g) = gG_{\omega(g)+} \in \gr G$ for $g \neq 1$.

To understand the $\F_{p}[\pi]$-Lie algebra, we need to find $[\sigma(g_{i}),\sigma(g_{j})] = \sigma\bigl( [g_{i},g_{j}] \bigr)$ for all $i,j = 1,\dotsc,d$. We recall from \cite{Sch} that $\sigma(g^{x}) = \overline{x}\pi^{v_{p}(x)} \act \sigma(g)$ for $g \in G \setminus\set{1}$ and $x \in \Z_{p}\setminus\set{0}$, where $\overline{x} = p^{-v_{p}(x)}x \pmod{p}$. Now, calculating $[g_{i},g_{j}]$ for all $i,j = 1,\dotsc,d$, we can find $x_{1},\dotsc,x_{d} \in \Z_{p}$ such that
\begin{equation*}
  [g_{i},g_{j}] = g_{1}^{x_{1}} \dotsb g_{d}^{x_{d}},
\end{equation*}
and thus
\begin{equation*}
  \bigl[ \sigma(g_{i}),\sigma(g_{j}) \bigr] = \sigma\bigl( [g_{i},g_{j}] \bigr) = \sum_{\ell=1}^{d} \overline{x}_{\ell}\pi^{v_{p}(x_{\ell})} \act \sigma(g_{\ell}).
\end{equation*}
See the proofs of \cite[Lem.~26.4 and Prop.~26.5]{Sch} for more details, and note that it is enough to only consider the $i<j$ cases.

Let $\set{\ell_{1},\dotsc,\ell_{r}}$ be the subset of $\set{1,\dotsc,d}$ such that $v_{p}(x_{\ell_{s}}) = 0$ and $v_{p}(x_{\ell}) > 0$ for $\ell \notin \set{\ell_{1},\dotsc,\ell_{r}}$, and recall that $\lie{g} = k \otimes_{\F_{p}[\pi]} \gr G$ has basis $\xi_{i} = 1 \otimes \sigma(g_{i})$. Since $\pi$ acts trivially on $k$ here, we see that
\begin{equation*}
  [\xi_{i},\xi_{j}] = \bigl[ 1\otimes\sigma(g_{i}),1\otimes\sigma(g_{j}) \bigr] = \sum_{s=1}^{r} \overline{x}_{\ell_{s}} \xi_{\ell_{s}}.
\end{equation*}
Now we have a basis $(\xi_{1},\dotsc,\xi_{d})$ of $\lie{g} = k \otimes_{\F_{p}[\pi]} \gr G$, and we know all the structure constants.

\begin{remark}\label{rem:struc-consts-lift}
  Note that the structure constants are in $\F_{p} \subseteq k$ by the above, so we can lift them to structure constants in $\set{0,1,\dotsc,p-1} \subseteq \Z$, which will be useful later. Also note that we will often (but not always) be able to lift $\lie{g}$ to a $\Z$-Lie algebra $\lie{g}_{\Z}$ with $\lie{g} = \lie{g}_{\Z} \otimes k$.
\end{remark}

Assume from now on that the Lie algebra $\lie{g}$ is unitary, i.e., that $[\xi_{i},\xi_{j}] = \sum_{\ell=1}^{d} c_{ij\ell} \xi_{\ell}$ has $\sum_{j=1}^{d} c_{ijj} = 0$. This will be the case for all Lie algebras, we will work with in this paper. Suppose furthermore that $\lie{g}$ is a graded Lie algebra, graded by finitely many positive integers, $\lie{g} = \lie{g}^{1} \oplus \lie{g}^{2} \oplus \dotsb \oplus \lie{g}^{m}$, which will also be the case for all Lie algebras we work with in this paper.

\begin{remark}\label{rem:g-Z-grading}
  Note that any $p$-valuable group $G$ admits a $p$-valuation $\omega$ with values in $\frac{1}{m}\Z$ for some $m \in \N$, cf.\ \cite[Cor.~33.3]{Sch}. Thus we can reindex the filtration of $G$ by letting $G^{i} = G_{\frac{i}{m}}$ for $i=0,1,\dotsc$, and this translates to $\gr^{i} G = \gr_{\frac{i}{m}} G$ and $\lie{g}^{i} = \lie{g}_{\frac{i}{m}}$ in general. In the cases we care about there will be no zero graded part, which allows us to make the above assumption.
\end{remark}

Then $\bigwedge^{n} \lie{g}$ is graded as well by letting
\begin{equation*}
  \gr^{j}\Bigl( \bigwedge^{n}\lie{g} \Bigr) = \bigoplus_{j_{1} + \dotsb + j_{n} = j} \lie{g}^{j_{1}} \wedge \dotsb \wedge \lie{g}^{j_{n}}.
\end{equation*}
We note that, since $\lie{g}$ is finite dimensional, there are only finitely many non-zero $\gr^{j}\bigl( \bigwedge^{n} \lie{g} \bigr)$ we are interested in, and we can find a basis with elements of the form $\xi_{i_{1}} \wedge \xi_{i_{2}} \wedge \dotsb \wedge \xi_{i_{n}}$ of each of these using our basis $(\xi_{1},\dotsc,\xi_{d})$ of $\lie{g}$.

\begin{remark}
  When ordering the basis of $\gr^{j} \bigl( \bigwedge^{n}\lie{g} \bigr) = \bigoplus_{j_{1} + \dotsb + j_{n} = j} \lie{g}^{j_{1}} \wedge \dotsb \wedge \lie{g}^{j_{n}}$, we will do it as follows. First we order the $\lie{g}^{j_{1}} \wedge \dotsb \wedge \lie{g}^{j_{n}}$ by the lexicographical order on $(j_{1},\dotsc,j_{n})$. Then we order the basis of each $\lie{g}^{j_{1}} \wedge \dotsb \wedge \lie{g}^{j_{n}}$ by the lexicographical order on equal $j_{\ell}$'s, e.g., if $\lie{g}^{1} = \Span_{k}(\xi_{1},\xi_{3})$ and $\lie{g}^{2} = \Span_{k}(\xi_{2},\xi_{4})$, then $\lie{g}^{1} \wedge \lie{g}^{1} \wedge \lie{g}^{2}$ has basis $\xi_{1} \wedge \xi_{3} \wedge \xi_{2}, \xi_{1} \wedge \xi_{3} \wedge \xi_{4}$.
\end{remark}

Assuming furthermore that $k$ is $\Z$-graded (concentrated in degree $0$), the space $\Hom_{k}\bigl( \bigwedge^{n}\lie{g}, k \bigr)$ inherits the $\Z$-grading
\[
  \Hom_{k}\Bigl( \bigwedge^n \lie{g}, k \Bigr) = \bigoplus_{s \in \Z} \Hom_{k}^s\Bigl( \bigwedge^n\lie{g}, k \Bigr),
\]
where $\Hom_{k}^s$ denotes the homogeneous $k$-linear maps of degree $s$, cf.\ \cite[Lem.~4.2]{Fossum}. We note, by \cite[Chap.~1~§3.7]{Fuks}, that these gradings on the chain and cochain complexes transfer to gradings on the homology and cohomology. We write
\begin{equation*}
  H^{s,t} = H^{s,t}(\lie{g}, k) = H^{s+t}\Bigl( \gr^s \Hom_{k}\bigl(\bigwedge^{\bullet} \lie{g}, k\bigr) \Bigr),
\end{equation*}%
and we let $e_{i_{1},i_{2},\dotsc,i_{n}}$ denote the dual of $\xi_{i_{1}}\wedge\xi_{i_{2}}\wedge\dotsb\wedge\xi_{i_{n}}$ and note that we get a dual basis of $\Hom_{k}^{s}(\bigwedge^{n}\lie{g},k)$ consisting of such elements.

\begin{remark}
  We do not spend effort to describe the homology for a few reasons. First, we need the cohomology, not the homology, in our spectral sequence. Second, by \cite[Chap.~1~§3.6]{Fuks}, we have a version of Poincaré duality for Lie algebra cohomology, i.e., $H^{n}(\lie{g},k) \iso H_{n-d}(\lie{g},k)$, so we can easily describe the homology using the cohomology. Third, we will care about the cup product later, and we do not get a nice product in homology, cf.\ \cite[Chap.~1~§3.2]{Fuks}.
\end{remark}

Now we have bases of all $\gr^{j}\bigl( \bigwedge^{n}\lie{g} \bigr)$, and by \cite[Chap.~1~§3.7]{Fuks} we get graded chain complexes that we can use to find the homology of $\lie{g}$. Here the boundary maps of
\[
  \begin{tikzcd}
    \cdots \ar[r,"d_{4}"] & \bigwedge^{3}\lie{g} \ar[r,"d_{3}"] & \bigwedge^{2}\lie{g} \ar[r,"d_{2}"] & \lie{g} \ar[r,"d_{1}"] & k \ar[r] &  0,
  \end{tikzcd}
\]
are given by
\begin{equation*}
  d_{n}(x_{1} \wedge \dotsb \wedge x_{n}) = \sum_{i<j} (-1)^{i+j}[x_{i},x_{j}]\wedge x_{1} \wedge \dotsb \wedge \widehat{x}_{i} \wedge \dotsb \wedge \widehat{x}_{j} \wedge \dotsb \wedge x_{n},
\end{equation*}
and the coboundary maps
\[
  \begin{tikzcd}
    \cdots & \ar[l,"\partial_{3}"'] \Hom_{k}\Bigl( \bigwedge^{2}\lie{g}, k \Bigr) & \ar[l,"\partial_{2}"'] \Hom_{k}(\lie{g},k) & \ar[l,"\partial_{1}"'] \Hom_{k}(k, k) = k &  \ar[l] 0,
  \end{tikzcd}
\]
are the dual maps of the boundary maps (see \cite[Chap.~1~§3.1]{Fuks} for more details). Thus, if we use the dual basis, of our basis of $\bigwedge^{n}\lie{g}$, in $\Hom_{k}\bigl( \bigwedge^{n} \lie{g},k \bigr)$, we get that $\partial_{n} = d_{n}^{\top}$ as matrices, where $(\edot)^{\top}$ is the transpose. Since we know bases and linear maps explicitly, and we know that the linear maps restrict to graded linear maps, we can now find matrices describing all graded linear maps
\begin{equation*}
  \gr^{j} \Bigl( \bigwedge^{n}\lie{g} \Bigr) \to \gr^{j} \Bigl( \bigwedge^{n-1}\lie{g} \Bigr),
\end{equation*}
and thus we can find matrices describing all graded linear maps
\begin{equation*}
  \Hom_{k}^{s}\Bigl( \bigwedge^{n-1}\lie{g},k \Bigr) \to \Hom_{k}^{s}\Bigl( \bigwedge^{n}\lie{g}, k \Bigr),
\end{equation*}
which allows us to exactly calculate the cohomology and find a basis for each cohomology space.

\begin{remark}
  We note here that when the structure constants can be lifted to $\Z$, it clear that the above matrices describing the (co)boundary maps can be lifted to $\Z$ by looking at the formula for the boundary maps. This allows us to do calculations over $\Z$ and just reduce mod $p$ in the end.
\end{remark}

Suppose now that we have found the cohomology $H^{s,t} = H^{s,t}(\lie{g},k)$ for all $s,t$. To get information about the (continuous) group cohomology $H^{n}(G,k)$, we need to use the multiplicative spectral sequence \eqref{eq:spec-sec-tech}, i.e.,
\begin{equation*}
  E_{1}^{s,t} = H^{s,t}(\lie{g}, k) \Longrightarrow H^{s+t}(G,k)
\end{equation*}
and information about spectral sequences in general. We already know from \cite{Sor} that this spectral sequence collapses at a finite page, and one can hope that it will actually collapse at the first page. One way we can verify that the spectral sequence (in certain cases) collapses at the first page, is by considering the exact bidegrees of the differentials. We know that the differentials $d_{r}^{s,t} \colon E_{r}^{s,t} \to E_{r}^{s+r,t+1-r}$ have bidegree $(r,1-r)$, so if the non-zero modules on the first pages is distributed in such a way that all differentials $d_{r}^{s,t}$ are trivial for all $s,t$ and $r\geq1$, then we can be sure that the spectral sequence collapses on the first page. This will become clearer when we look at examples in the next few sections.

\begin{remark}
  We note here that the spectral sequence will collapse on the first page for the pro-$p$ Iwahori subgroups of $\SL_{2}(\Q_{p})$ and $\SL_{3}(\Q_{p})$ by the above argument, but in \cite{thesis} it is shown that this argument does not work for the pro-$p$ Iwahori subgroup of $\SL_{4}(\Q_{p})$.
\end{remark}

Now note, by \cite[Chap.~1~§3.7]{Fuks}, that the cup product is compatible with the gradings on the Lie algebra cohomology, in particular
\begin{equation}
H^{s,t} \cup H^{s',t'} \subseteq H^{s+s',t+t'},\label{eq:Hst-Hst-subset}
\end{equation}
where $H^{s,t} = H^{s,t}(\lie{g},k)$. Thus, since the spectral sequence is multiplicative, we can describe the cup product on $H^{*}(G,k)$ when the spectral sequence collapses on the first page. Some cup products will be trivially zero by \eqref{eq:Hst-Hst-subset}, and for the rest of the cup products, we can calculate them with an explicit basis using the following.

For $f \in \Hom_{k}\bigl( \bigwedge^{p}\lie{g}, k \bigr)$ and $g \in \Hom_{k}\bigl( \bigwedge^{q}\lie{g}, k \bigr)$, we know from \cite[Chap.~XIII, Sect.~8]{CartanHomAlg}, that the cup product in cohomology is induced by: $f \cup g \in \Hom_{k}\bigl( \bigwedge^{p+q}\lie{g}, k \bigr)$ defined by
\begin{equation}
  \label{eq:cup-prod-def}
  (f \cup g)(x_{1} \wedge \dotsb \wedge x_{p+q})  = \sum_{\mathclap{\substack{ \sigma \in S_{p+q} \\ \sigma(1) < \dotsb < \sigma(p) \\ \sigma(p+1) < \dotsb < \sigma(p+q) }}} \sign(\sigma) f(x_{\sigma(1)} \wedge \dotsb \wedge x_{\sigma(p)}) g(x_{\sigma(p+1)} \wedge \dotsb \wedge x_{\sigma(p+q)}).
\end{equation}
With our basis and dual basis this clearly reduces to
\begin{equation}
  \label{eq:cup-prod-basis}
  (e_{i_{1},\dotsc,i_{p}} \cup e_{j_{1},\dotsc,j_{q}})(\xi_{k_{1}} \wedge \dotsb \wedge \xi_{k_{p+q}}) = \sign(\sigma)\bm{1}_{\set{i_{1},\dotsc,i_{p}}\cup\set{j_{1},\dotsc,j_{q}} = \set{k_{1},\dotsc,k_{p+q}}},
\end{equation}
where $\bm{1}_{B}$ is the indicator function that is $1$ if $B$ is true and $0$ otherwise and $\sigma \in S_{p+q}$ is the permutation with $\sigma(k_{\ell}) = i_{\ell}$ and $\sigma(k_{p+\ell}) = j_{\ell}$. (Here we assume that the wedge product $\xi_{k_{1}} \wedge \dotsc \wedge \xi_{k_{p+q}}$ are already ordered in the lexicographical order we are working with.) This equation clearly simplifies calculations a lot, and as long as we work with explicit bases for everything, it makes it quite possible to calculate cup products explicitly. In particular, we will by the above have an explicit basis $( v_{1}^{(s,t)},\dotsc, v_{m}^{(s,t)} )$ of each $H^{s,t}$, where each $v_{m}^{(s,t)}$ can be lifted to a linear combination of the $e_{i_{1},\dotsc, i_{n}}$'s, which makes the calculations easier.

Finally note that even when the spectral sequence does not (necessarily) collapse on the first page, we can still get some bounds on the dimensions of $H^{n}(G,k)$ that will allow us to draw some conclusions about the structure of $H^{*}(G,k)$. This is investigated in more detail in \cite{thesis}, but we will not explore it in this paper.

In the following two sections of this paper we will focus on using the techniques described in this section to get as much as possible information about the cohomology of $H^{*}(I,k)$, where $I$ is the pro-$p$ Iwahori subgroup of $\SL_{n}(\Q_{p})$ for $n=2$ or $n=3$.


\section{\texorpdfstring{$H^{*}(I,\F_{p})$}{H*(I,Fp)} for the pro-\texorpdfstring{$p$}{p} Iwahori \texorpdfstring{$I \subseteq \SL_{2}(\Z_{p})$}{I in SL2(Zp)}}%
\label{sec:Iwa-SL2}

In this section we will describe the continuous group cohomology of the pro-$p$ Iwahori subgroup $I$ of $\SL_{2}(\Q_{p})$.

When $I$ is the pro-$p$ Iwahori subgroup in $\SL_{2}(\Q_{p})$, we know by \Cref{subsec:setup-iwa} that we can take it to be of the form
\begin{equation*}
  I = \pmat{1+p\Z_{p} & \Z_{p} \\ p\Z_{p} & 1+p\Z_{p}}^{\!\!\det = 1} \subseteq \SL_{2}(\Z_{p}).
\end{equation*}
In this case, an obvious guess for an ordered basis (using that $(1+p)^{\Z_{p}} = 1+p\Z_{p}$) is
\begin{align*}
  g_{1}' &= \pmat{1 & 0 \\ p & 1}, & g_{2}' &= \pmat{1+p & 0 \\ 0 & (1+p)^{-1}}, & g_{3}' &= \pmat{1 & 1 \\ 0 & 1}.
\end{align*}
Because we want to be able to describe the commutators using this ordered basis, we will at one point need to solve for $x$ in equation of the form $(1+p)^{x} = y$. For this reason a better choice of ordered basis is (as described in \Cref{subsec:setup-iwa})
\begin{align}
  \label{eq:gis-SL2}
  g_{1} &= \pmat{1 & 0 \\ p & 1}, & g_{2} &= \pmat{\exp(p) & 0 \\ 0 & \exp(-p)}, & g_{3} &= \pmat{1 & 1 \\ 0 & 1}.
\end{align}
In this case the above equations to solve translate to solving for $x$ in $\exp(px) = y$, which we can easily do, as $x = \frac{1}{p}\log(y)$.

\subsection{Finding the commutators \texorpdfstring{$[\xi_{i},\xi_{j}]$}{[xi-i,xi-j]}}%
\label{subsec:non-id-xi_ij-SL2}

Now write
\begin{equation}
  \label{eq:gixi-SL2}
  g_{1}^{x_{1}}g_{2}^{x_{2}}g_{3}^{x_{3}} = \pmat{\exp(px_{2}) & x_{3}\exp(px_{2}) \\ px_{1}\exp(px_{2}) & px_{1}x_{3}\exp(px_{2}) + \exp(-px_{2})} = \pmat{a_{11} & a_{12} \\ a_{21} & a_{22}}.
\end{equation}
Furthermore, write $g_{ij} = [g_{i},g_{j}]$ and $\xi_{ij} = [\xi_{i},\xi_{j}]$. Then we are ready to find $x_{1},x_{2},x_{3}$ such that $g_{ij} = g_{1}^{x_{1}}g_{2}^{x_{2}}g_{3}^{x_{3}}$ for different $i<j$. (In the following we use that $\frac{1}{1-p} = 1 + p + p^{2} + \dotsb$ and $\log(1-p) = -p - \frac{p^{2}}{2} - \frac{p^{3}}{3} - \dotsb$.)

We now list all non-identitiy commutators $g_{ij} = [g_{i},g_{j}]$ and find $\xi_{ij} = [\xi_{i},\xi_{j}]$ based on these. (For $g_{ij} = 1_{2}$ it is clear that $x_{1} = x_{2} = x_{3} = 0$, and thus $\xi_{ij} = 0$.)

\begin{description}
  \item[$g_{12} = \pmat{ 1 & 0 \\ p\bigl( 1 - \exp(-2p) \bigr) & 1 }$] Comparing $g_{12}$ with \eqref{eq:gixi-SL2}, we see that $x_{2} = x_{3} = 0$. This leaves $a_{21} = px_{1} = p\bigl( 1 - \exp(-2p) \bigr) = 2p^{2} + O(p^{3})$, which implies that $x_{1} = 2p + O(p^{2})$. Hence $\sigma(g_{12}) = 2\pi \act \sigma(g_{1})$, which implies that $\xi_{12} = 0$.

  \item[$g_{13} = \pmat{ 1-p & p \\ -p^{2} & 1+p+p^{2} }$] Comparing $g_{13}$ with \eqref{eq:gixi-SL2}, we see that
        \begin{align*}
          a_{11} &= \exp(px_{2}) = 1-p, \\
          a_{12} &= x_{3}\exp(px_{2}) = x_{3}(1-p) = p, \\
          a_{21} &= px_{1}\exp(px_{2}) = px_{1}(1-p) = -p^{2},
        \end{align*}
        and thus
        \begin{align*}
          x_{2} &= \dfrac{1}{p}\log(1-p) = \dfrac{1}{p}\bigl( (-p) + O(p^{2}) \bigr) = -1 + O(p), \\
          x_{3} &= \dfrac{p}{1-p} = p + O(p^{2}), \\
          x_{1} &= \dfrac{-p^{2}}{p(1-p)} = -p + O(p^{2}).
        \end{align*}
        Hence $\sigma(g_{13}) = -\pi \act \sigma(g_{1}) - \sigma(g_{2}) - \pi \act \sigma(g_{3})$, which implies that $\xi_{13} = -\xi_{2}$.

  \item[$g_{23} = \pmat{ 1 & \exp(2p)-1 \\ 0 & 1 }$] Comparing $g_{23}$ with \eqref{eq:gixi-SL2}, we see that $x_{1} = x_{2} = 0$. This leaves $a_{12} = x_{3} = \exp(2p)-1 = 2p + O(p^{2})$. Hence $\sigma(g_{23}) = 2\pi \act \sigma(g_{3})$, which implies that $\xi_{23} = 0$.
\end{description}

To clarify, we found that
\begin{align*}
  \sigma(g_{12}) &= 2\pi \act \sigma(g_{1}), \\
  \sigma(g_{13}) &= -\pi \act \sigma(g_{1}) - \sigma(g_{2}) - \pi \act \sigma(g_{3}), \\
  \sigma(g_{23}) &= 2\pi \act \sigma(g_{3}),
\end{align*}
and recalling that $\xi_{i} = 1 \otimes \sigma(g_{i})$ in $\F_{p} \otimes_{\F_{p}[\pi]} \gr G$, where $\pi$ acts trivially on $\F_{p}$, we get that
\begin{align}\label{eq:xi_ij-SL2}
  \xi_{12} &= 0, & \xi_{13} &= -\xi_{2}, & \xi_{23} &= 0,
\end{align}
where $\xi_{ij} = [\xi_{i},\xi_{j}]$.

\subsection{Describing the graded chain complex, \texorpdfstring{$\gr^{j}\bigl(\bigwedge^{n}\lie{g}\bigr)$}{grj(wedge-n g)}}%
\label{subsec:graded-complex-SL2}

Looking at \eqref{eq:Iwa-p-val-basis-SLn} (with $e=1$ and $h=2$), we see that
\begin{align*}
  \omega(g_{1}) &= 1-\frac{1}{2} = \frac{1}{2}, & \omega(g_{2}) &= 1, & \omega(g_{3}) &= \frac{1}{2}.
\end{align*}
Hence $\lie{g}^{1} = \lie{g}_{\frac{1}{2}} = \Span_{\F_{p}}(\xi_{1},\xi_{3})$ and $\lie{g}^{2} = \lie{g}_{1} = \Span_{\F_{p}}(\xi_{2})$, cf.\ \Cref{rem:g-Z-grading}.


Now we are ready to describe the graded chain complex
\begin{equation*}
  \gr^{j}\Bigl( \bigwedge^{n}\lie{g} \Bigr) = \bigoplus_{j_{1} + \dotsb + j_{n} = j} \lie{g}^{j_{1}} \wedge \dotsb \wedge \lie{g}^{j_{n}}
\end{equation*}
and its bases. We list the grading of $\bigwedge^{n}\lie{g}$ for all $n$.

\begin{description}
  \item[$n=0$]
        \begin{equation*}
          \gr^{j}(\F_{p}) =
          \begin{dcases}
            \F_{p} & j=0, \\
            0 & \text{otherwise.}
          \end{dcases}
        \end{equation*}
        Bases:
        \begin{align*}
          \F_{p}: \quad & 1.
        \end{align*}

   \item[$n=1$]
        \begin{equation*}
          \gr^{j}(\lie{g}) =
          \begin{dcases}
            \lie{g}^{2} & j=2, \\
            \lie{g}^{1} & j=1, \\
            0          & \text{otherwise.}
          \end{dcases}
        \end{equation*}
        Bases:
        \begin{align*}
          \lie{g}^{2}: \quad & \xi_{2}, \\
          \lie{g}^{1}: \quad & \xi_{1}, \xi_{3}.
        \end{align*}

   \item[$n=2$]
        \begin{equation*}
          \gr^{j}\Bigl( \bigwedge^{2}\lie{g} \Bigr) =
          \begin{dcases}
            \lie{g}^{1} \wedge \lie{g}^{2} & j=3, \\
            \lie{g}^{1} \wedge \lie{g}^{1} & j=2, \\
            0                             & \text{otherwise.}
          \end{dcases}
        \end{equation*}
        Bases:
        \begin{align*}
          \lie{g}^{1} \wedge \lie{g}^{2}: \quad & \xi_{1} \wedge \xi_{2}, \xi_{3} \wedge \xi_{2}, \\
          \lie{g}^{1} \wedge \lie{g}^{1}: \quad & \xi_{1} \wedge \xi_{3}.
        \end{align*}

  \item[$n=3$]
        \begin{equation*}
          \gr^{j}\Bigl( \bigwedge^{3}\lie{g} \Bigr) =
          \begin{dcases}
            \lie{g}^{1} \wedge \lie{g}^{1} \wedge \lie{g}^{2} & j=4, \\ 0                                                & \text{otherwise.}
          \end{dcases}
        \end{equation*}
        Bases:
        \begin{align*}
          \lie{g}^{1} \wedge \lie{g}^{1} \wedge \lie{g}^{2}: \quad & \xi_{1} \wedge \xi_{3} \wedge \xi_{2}.
        \end{align*}

   \item[$n\geq4$]
        \begin{equation*}
          \gr^{j}\Bigl( \bigwedge^{n}\lie{g} \Bigr) = 0 \text{ for all } j.
        \end{equation*}
\end{description}

\begin{table}[ht]
  \centering
  \caption{Dimensions of \texorpdfstring{$\gr^{j}\bigl( \bigwedge^{n} \lie{g} \bigr)$}{grade j of n wedges of g}  for the \texorpdfstring{$I \subseteq \SL_{2}(\Z_{p})$}{I in SL2(Zp)} case.}
  \label{tab:graded-dims-SL2}
  $\begin{NiceArray}{*{6}{c}}[hvlines]
    \diagbox{n}{j} & 0 & 1 & 2 & 3 & 4 \\
    0 & 1 \\
    1 & & 2 & 1 \\
    2 & & & 1 & 2 \\
    3 & & & & & 1
  \end{NiceArray}$
\end{table}

We collect the above information about the dimensions of the chain complex of $\lie{g}$ in \Cref{tab:graded-dims-SL2}, and note that we only need to consider non-zero (non-empty) entries of the table, when we calculate  $H^{s,t} = H^{s,n-s}$ (where $H^{s,t} = H^{s,t}(\lie{g},\F_{p})$). Also, recalling that
\begin{equation*}
  \Hom_{\F_{p}}\Bigl( \bigwedge^{n}\lie{g}, \F_{p} \Bigr) = \bigoplus_{s \in \Z} \Hom_{\F_{p}}^{s}\Bigl( \bigwedge^{n}\lie{g}, \F_{p} \Bigr),
\end{equation*}
we see that, with $j=-s$, we get the same table for dimensions of the graded hom-spaces in the cochain complex.

\subsection{Finding the graded Lie algebra cohomology, \texorpdfstring{$H^{s,t}(\lie{g},\F_{p})$}{H(s,t)(g,\F_{p})}}%
\label{subsec:graded-coh-SL2}

We will now go through all different graded chain complexes one by one, using that $\gr^{j}$ in the chain complex corresponds to $\gr^{s}$ with $s = -j$ in the cochain complex. We note that the graded chain complex corresponds to vertical downwards arrows in \Cref{tab:graded-dims-SL2}, while the cochain complex corresponds to vertical upwards arrows. Finally, we reiterate that $H^{n} = H^{n}(\lie{g},\F_{p})$ and $H^{s,t} = H^{s,t}(\lie{g},\F_{p})$ in the following.

In grade $0$ we have the chain complex
\[
  \begin{tikzcd}
    0 \ar[r] & \F_{p} \ar[r] & 0,
  \end{tikzcd}
\]
which gives us the grade $0$ cochain complex
\[
  \begin{tikzcd}
    0 & \ar[l] \Hom_{\F_{p}}^{0}(\F_{p},\F_{p}) & \ar[l] 0.
  \end{tikzcd}
\]
So $H^{0} = H^{0,0}$ with $\dim H^{0,0} = 1$.

In grade $1$ we have the chain complex
\[
  \begin{tikzcd}
    0 \ar[r] & \lie{g}^{1} \ar[r] & 0,
  \end{tikzcd}
\]
which gives us the grade $-1$ cochain complex
\[
  \begin{tikzcd}
    0 & \ar[l] \Hom_{\F_{p}}^{-1}(\lie{g},\F_{p}) & \ar[l] 0.
  \end{tikzcd}
\]
So $\dim H^{-1,2} = 2$ by \Cref{tab:graded-dims-SL2}.

In grade $2$ we have the chain complex
\[
  \begin{tikzcd}[ampersand replacement=\&]
    0 \ar[r] \& \lie{g}^{1} \wedge \lie{g}^{1} \ar[r, "{(1)}" {yshift=2pt}] \& \lie{g}^{2} \ar[r] \& 0,
  \end{tikzcd}
\]
since
\begin{align*}
  \lie{g}^{1} \wedge \lie{g}^{1} &\to \lie{g}^{2} \\
  \xi_{1} \wedge \xi_{3} &\mapsto -[\xi_{1},\xi_{3}] = \xi_{2}.
\end{align*}
This gives us the grade $-2$ cochain complex
\[
  \begin{tikzcd}[ampersand replacement=\&]
    0 \& \ar[l] \Hom_{\F_{p}}^{-2}\bigl( \bigwedge^{2} \lie{g}, \F_{p} \bigr) \& \ar[l, "{(1)}"' {yshift=2pt}] \Hom_{\F_{p}}^{-2}(\lie{g},\F_{p}) \& \ar[l] 0.
  \end{tikzcd}
\]
So with $d = (1)$, and comparing with \Cref{tab:graded-dims-SL2},
\begin{align*}
  \dim H^{-2,3} &= \dim \kernel(d) = 0, \\
  \dim H^{-2,4} &= \dim \coker(d) = 0.
\end{align*}

In grade $3$ we have the chain complex
\[
  \begin{tikzcd}
    0 \ar[r] & \lie{g}^{1} \wedge \lie{g}^{2} \ar[r] & 0,
  \end{tikzcd}
\]
which gives us the grade $-3$ cochain complex
\[
  \begin{tikzcd}
    0 & \ar[l] \Hom_{\F_{p}}^{-3}\bigl( \bigwedge^{2}\lie{g}, \F_{p} \bigr) & \ar[l] 0.
  \end{tikzcd}
\]
So $\dim H^{-3,5} = 2$ by \Cref{tab:graded-dims-SL2}.

In grade $4$ we have the chain complex
\[
  \begin{tikzcd}
    0 \ar[r] & \lie{g}^{1} \wedge \lie{g}^{1} \wedge \lie{g}^{2}  \ar[r] & 0,
  \end{tikzcd}
\]
which gives us the grade $-4$ cochain complex
\[
  \begin{tikzcd}
    0 & \ar[l] \Hom_{\F_{p}}^{-4}\bigl( \bigwedge^{3}\lie{g}, \F_{p} \bigr) & \ar[l] 0.
  \end{tikzcd}
\]
So $\dim H^{-4,7} = 1$ by \Cref{tab:graded-dims-SL2}.

\begin{table}[ht]
  \centering
  \caption{Dimensions of \texorpdfstring{$E_{1}^{s,t} = H^{s,t}(\lie{g},\F_{p})$}{E1(s,t) = H(s,t)(g,Fp)} for the \texorpdfstring{$I \subseteq \SL_{2}(\Z_{p})$}{I in SL2(Zp)} case.}
  \label{tab:graded-coh-dims-SL2}
  \renewcommand{\arraystretch}{1.5}
  $\begin{NiceArray}{*{6}{c}}[hvlines, columns-width=auto]
    \diagbox{t}{s} & 0 & -1 & -2 & -3 & -4 \\
    0 & 1 \\
    1 & \\
    2 & & 2 \\
    3 & \\
    4 & \\
    5 & & & & 2 \\
    6 & \\
    7 & & & & & 1
  \end{NiceArray}$
  \renewcommand{\arraystretch}{1}
\end{table}

Altogether, we see that
\begin{equation}
  \label{eq:Hn-to-Hst-SL2}
  \begin{aligned}
    H^{0} &= H^{0,0}, \\
    H^{1} &= H^{-1,2}, \\
    H^{2} &= H^{-3,5}, \\
    H^{3} &= H^{-4,7},
  \end{aligned}
\end{equation}
with dimension as described in \Cref{tab:graded-coh-dims-SL2}.

\subsection{Describing the group cohomology, \texorpdfstring{$H^{n}(I,\F_{p})$}{Hn(I,\F_{p})}}%
\label{subsec:group-coh-SL2}

We note that all differentials $d_{r}^{s,t} \colon E_{r}^{s,t} \to E_{r}^{s+r,t+1-r}$ in \Cref{tab:graded-coh-dims-SL2} has bidegree $(r,1-r)$, i.e., they are all below the $(r,-r)$ arrow going $r$ to the left and $r$ up in the table, where $r \geq 1$. Looking at \Cref{tab:graded-coh-dims-SL2}, this clearly means that all differentials for $r \geq 1$ are trivial, and thus the spectral sequence collapses on the first page. Hence $H^{s,t}(\lie{g},\F_{p}) = E_{1}^{s,t} \iso E_{\infty}^{s,t} = \gr^{s} H^{s+t}(I,\F_{p})$, and by \eqref{eq:Hn-to-Hst-SL2} and \Cref{tab:graded-coh-dims-SL2} we get that
\begin{equation}
  \label{eq:dim-HnI-SL2}
  \dim H^{n}(I,\F_{p}) =
  \begin{dcases}
    1 & n=0, \\
    2 & n=1, \\
    2 & n=2, \\
    1 & n=3.
  \end{dcases}
\end{equation}

Recalling that the spectral sequence is multiplicative, we also note, by \Cref{tab:graded-coh-dims-SL2}, that $H^{s,t} \cup H^{s',t'} \subseteq H^{s+s',t+t'}$ implies that the cup products
\begin{equation*}
  \gr^{s} H^{n}(I,\F_{p}) \otimes \gr^{s'} H^{n'}(I,\F_{p}) \to \gr^{s+s'} H^{n+n'}(I,\F_{p})
\end{equation*}
are trivial, except for the obvious ones with $H^{0}(I,\F_{p})$ and $H^{1} \otimes H^{2} \to H^{3}$. We now want to describe the cup product $H^{1} \otimes H^{2} \to H^{3}$.

Let $e_{i_{1},\dotsc,i_{m}} = (\xi_{i_{1}} \wedge \dotsb \wedge \xi_{i_{m}})^{*}$ be the element of the dual basis of $\Hom_{\F_{p}}\bigl( \bigwedge^{m}\lie{g},\F_{p} \bigr)$ corresponding to $\xi_{i_{1}} \wedge \dotsb \wedge \xi_{i_{m}}$ in the basis of $\bigwedge^{m}\lie{g}$. Looking at the cochain complexes and descriptions of the maps above together with the known bases of the graded chain complexes, we get the following precise descriptions of the of the graded cohomology spaces $H^{s,t} = H^{s,t}(\lie{g},\F_{p})$
\begin{equation}\label{eq:Hst-spaces-SL2}
  \begin{aligned}
    H^{1} = H^{-1,2} &= \Span_{\F_{p}}\bigl( v_{1}^{(-1,2)}, v_{2}^{(-1,2)} \bigr), \\
    H^{2} = H^{-3,5} &= \Span_{\F_{p}}\bigl( v_{1}^{(-3,5)}, v_{2}^{(-3,5)} \bigr), \\
    H^{3} = H^{-4,7} &= \Span_{\F_{p}}\bigl( v_{1}^{(-4,7)} \bigr),
  \end{aligned}
\end{equation}
where $v_{1}^{(-1,2)}$ lifts to $e_{1}$ in $\Hom_{\F_{p}}^{-1}(\lie{g},\F_{p})$ and $v_{2}^{(-1,2)}$ lifts to $e_{3}$ in $\Hom_{\F_{p}}^{-1}(\lie{g},\F_{p})$, $v_{1}^{(-3,5)}$ lifts to $e_{1,2}$ in $\Hom_{\F_{p}}^{-3}(\bigwedge^{2}\lie{g},\F_{p})$ and $v_{2}^{(-3,5)}$ lifts to $e_{3,2}$ in $\Hom_{\F_{p}}^{-3}(\bigwedge^{2}\lie{g},\F_{p})$, and $v_{1}^{(-4,7)}$ lifts to $e_{1,3,2}$ in $\Hom_{\F_{p}}^{-4}(\bigwedge^{3}\lie{g},\F_{p})$.

\begin{remark}
  Generally the cohomology spaces will be much harder to describe and we would get $v_{1}^{(s,t)},\dotsc,v_{m}^{(s,t)}$ for a basis of $H^{s,t}$, where each $v_{j}^{(s,t)}$ lifts to a linear combination of elements of the form $e_{i_{1},\dotsc,i_{k}}$. This will be the case in \Cref{sec:Iwa-SL3}.
\end{remark}

For $f \in \Hom_{\F_{p}}\bigl( \bigwedge^{p}\lie{g}, \F_{p} \bigr)$ and $g \in \Hom_{\F_{p}}\bigl( \bigwedge^{q}\lie{g}, \F_{p} \bigr)$, we recall from \eqref{eq:cup-prod-def} that the cup product in cohomology is induced by: $f \cup g \in \Hom_{\F_{p}}\bigl( \bigwedge^{p+q}\lie{g}, \F_{p} \bigr)$ defined by
\begin{equation*}
  (f \cup g)(x_{1} \wedge \dotsb \wedge x_{p+q})  = \sum_{\mathclap{\substack{ \sigma \in S_{p+q} \\ \sigma(1) < \dotsb < \sigma(p) \\ \sigma(p+1) < \dotsb < \sigma(p+q) }}} \sign(\sigma) f(x_{\sigma(1)} \wedge \dotsb \wedge x_{\sigma(p)}) g(x_{\sigma(p+1)} \wedge \dotsb \wedge x_{\sigma(p+q)}).
\end{equation*}
\begin{remark}
  We could just use \eqref{eq:cup-prod-basis} here, but we instead use the below as an example to clarify how we get \eqref{eq:cup-prod-basis} from \eqref{eq:cup-prod-def}.
\end{remark}
So, when finding
\[
  \begin{tikzcd}
    H^{-1,2} \otimes H^{-3,5} \ar[r,"\cup"] & H^{-4,7},
  \end{tikzcd}
\]
we need to calculate $e_{1} \cup e_{1,2}$, $e_{1} \cup e_{3,2}$, $e_{3} \cup e_{1,2}$ and $e_{3} \cup e_{3,2}$ on the basis $\basis = (\xi_{1} \wedge \xi_{3} \wedge \xi_{2})$ of $\gr^{4} \bigwedge^{3} \lie{g}$.

We first note that \eqref{eq:cup-prod-def} simplifies to
\begin{equation*}
  (e_{i} \cup e_{j,k})(x_{1} \wedge x_{2} \wedge x_{3}) = \sum_{\substack{\sigma \in S_{3} \\ \sigma(2) < \sigma(3)}} \sign(\sigma) e_{i}(x_{\sigma(1)}) e_{j,k}(x_{\sigma(2)} \wedge x_{\sigma(3)})
\end{equation*}
in these cases. Here the terms of the sum on the right is only non-zero if $x_{\sigma(1)} = \xi_{i}$ and $x_{\sigma(2)} \wedge x_{\sigma(3)} = \xi_{j} \wedge \xi_{k}$ (up to constants). In the case $e_{1} \cup e_{1,2}$ (resp.\ $e_{3} \cup e_{3,2}$), we can only have this if $x_{1} \wedge x_{2} \wedge x_{3}$ contains two copies of $\xi_{1}$ (resp.\ $\xi_{3}$), which implies that $x_{1} \wedge x_{2} \wedge x_{3} = 0$. So $e_{1} \cup e_{1,2} = 0$ and $e_{3} \cup e_{3,2} = 0$. Alternatively, one can see this by plugging in $x_{1} \wedge x_{2} \wedge x_{3} = \xi_{1} \wedge \xi_{3} \wedge \xi_{2}$ and simply calculating the right side.

In the case $e_{1} \cup e_{3,2}$, \eqref{eq:cup-prod-def} simplifies to
\begin{equation*}
  (e_{1} \cup e_{3,2})(x_{1} \wedge x_{2} \wedge x_{3}) = \sum_{\substack{\sigma \in S_{3} \\ \sigma(2) < \sigma(3)}} \sign(\sigma) e_{1}(x_{\sigma(1)}) e_{3,2}(x_{\sigma(2)} \wedge x_{\sigma(3)}).
\end{equation*}
When $x_{1} \wedge x_{2} \wedge x_{3} = \xi_{1} \wedge \xi_{3} \wedge \xi_{2}$, we see that the terms on the right side are only non-zero if $x_{\sigma(1)} = \xi_{1}$, i.e., $\sigma(1)=1$, and thus $\sigma = (1)$ since $\sigma(2) < \sigma(3)$. So $x_{\sigma(1)} = \xi_{1}$, $x_{\sigma(2)} \wedge x_{\sigma(3)} = \xi_{3} \wedge \xi_{2}$ and $\sign(\sigma) = 1$, which gives us $(e_{1} \cup e_{3,2})(\xi_{1} \wedge \xi_{3} \wedge \xi_{2}) = 1$. Hence $e_{1} \cup e_{3,2} = e_{1,3,2}$, and therefore \[ v_{1}^{(-1,2)} \cup v_{2}^{(-3,5)} = v_{1}^{(-4,7)}. \]

In the case $e_{3} \cup e_{1,2}$, \eqref{eq:cup-prod-def} simplifies to
\begin{equation*}
  (e_{3} \cup e_{1,2})(x_{1} \wedge x_{2} \wedge x_{3}) = \sum_{\substack{\sigma \in S_{3} \\ \sigma(2) < \sigma(3)}} \sign(\sigma) e_{3}(x_{\sigma(1)}) e_{1,2}(x_{\sigma(2)} \wedge x_{\sigma(3)}).
\end{equation*}
When $x_{1} \wedge x_{2} \wedge x_{3} = \xi_{1} \wedge \xi_{3} \wedge \xi_{2}$, we see that the terms on the right side are only non-zero if $x_{\sigma(1)} = \xi_{3}$, i.e., $\sigma(1)=2$, and thus $\sigma = (1,2)$ since $\sigma(2) < \sigma(3)$. So $x_{\sigma(1)} = \xi_{3}$, $x_{\sigma(2)} \wedge x_{\sigma(3)} = \xi_{1} \wedge \xi_{2}$ and $\sign(\sigma) = -1$, which gives us $(e_{3} \cup e_{1,2})(\xi_{1} \wedge \xi_{3} \wedge \xi_{2}) = -1$. Hence $e_{3} \cup e_{1,2} = -e_{1,3,2}$, and therefore \[ v_{2}^{(-1,2)} \cup v_{1}^{(-3,5)} = -v_{1}^{(-4,7)}. \]

In conclusion, all the non-trivial and non-zero cup products (up to graded commutativity) are:
\begin{equation}
  \label{eq:cup-products-SL2}
  \begin{aligned}
    v_{1}^{(-1,2)} \cup v_{2}^{(-3,5)} &= v_{1}^{(-4,7)}, \\
    v_{2}^{(-1,2)} \cup v_{1}^{(-3,5)} &= -v_{1}^{(-4,7)}.
  \end{aligned}
\end{equation}

Now, since the spectral sequence collapses on the first page, all of the above work on the cup product of the Lie algebra cohomology transfers to the cup product on $H^{*}(I,\F_{p})$ as described above. In particular, since all $H^{n}(I,\F_{p})$ only have one graded component, this is a clear description of the cup product on $H^{*}(I,\F_{p})$, and not just a graded cup product. Thus we have now shown the following.

\begin{theorem}
  Let $I$ be the pro-$p$ Iwahori subgroup of $\SL_{2}(\Q_{p})$. Then
  \begin{equation*}
    \dim H^{n}(I,\F_{p}) =
    \begin{dcases}
      1 & n=0, \\
      2 & n=1, \\
      2 & n=2, \\
      1 & n=3,
    \end{dcases}
  \end{equation*}
  and the only interesting cup product that is non-trivial is $H^{1}(I,\F_{p}) \times H^{2}(I,\F_{p}) \to H^{3}(I,\F_{p})$, which can be described by considering $H^{2}(I,\F_{p})$ the dual space of $H^{1}(I,\F_{p})$.
\end{theorem}

\begin{remark}\label{rem:quaternion}
  Let $D$ be the division quaternion algebra over $\Q_{p}$ for a prime $p>3$ and let $G = (1+\idm_{D})^{\Nrd = 1}$, where $\Nrd = \Nrd_{D/\Q_{p}}$ is the norm form. From \cite[Sect.~6.3]{Sor} (or from \cite[Prop.~7]{Henn}) we know that there is an isomorphism
  \begin{equation*}
    H^{*}(G,\F_{p}) \iso \F_{p} \oplus \F_{D} \oplus \F_{D} \oplus \F_{p}
  \end{equation*}
  of graded $\F_{p}$-algebras (where $\F_{D} \iso \F_{p^{2}}$ is viewed simply as a $\F_{p}$-vector space). I.e., $H^{n}(G,\F_{p})$ has the same dimensions as described in \eqref{eq:dim-HnI-SL2}. \cite{Sor} also shows that the only non-trivial and non-zero cup product is $H^{1}(G,\F_{p}) \times H^{2}(G,\F_{p}) \to H^{3}(G,\F_{p})$, which corresponds to the trace pairing $\F_{D} \times \F_{D} \to \F_{p}, (x,y) \mapsto \Tr(xy)$ (where $\Tr = \Tr_{\F_{D}/\F_{p}}$ from \cite[Def.~2.5]{Neukirch}). To be more explicit, let's assume that $p \equiv 3 \pmod{4}$. Then $x^{2}+1$ is irreducible over $\F_{p}$, so we can write $\F_{D} = \F_{p}[\alpha]$ with $\alpha^{2} = -1$, where $\F_{p}[\alpha]$ has $\F_{p}$-basis $(1,\alpha)$. Now, considering the maps $1\colon a+b\alpha \mapsto a+b\alpha$, $\alpha\colon a+b\alpha \mapsto -b + a\alpha$ and $\alpha^{2}=-1 \colon a+b\alpha \mapsto -a-b\alpha$, we see that the trace pairing is given by
  \begin{align*}
    \F_{D} \times \F_{D} &\to \F_{p} \\
    (1,1) &\mapsto \Tr(1) = 2, \\
    (1,\alpha) &\mapsto \Tr(\alpha) = 0, \\
    (\alpha,1) &\mapsto \Tr(\alpha) = 0, \\
    (\alpha,\alpha) &\mapsto \Tr(\alpha^{2}) = -2.
  \end{align*}
  This is (up to a multiple of $2$) the same as the description of the cup product on $H^{*}(I,\F_{p})$ above for $I \subseteq \SL_{2}(\Z_{p})$, so an interesting question is: Is there a nice relation between mod $p$ representations of $G = (1+\idm_{D})^{\Nrd = 1}$ and $I$? We already have bijections between \emph{certain} mod $p$ representations of $D^{\times}$ and $\GL_{2}(\Q_{p})$ from the Jacquet-Langlands correspondence for $\GL_{2}$ (cf.\ \cite{JL}), but by \cite[Rem.~4.5]{JL-remark} irreducible representations of $D^{\times}$ are trivial on $1+\idm_{D}$, so we need something new if we want a correspondence between mod $p$ representations of $G = (1+\idm_{D})^{\Nrd = 1}$ and $I \subseteq \SL_{2}(\Q_{p})$.

  As a continuation to the question in the $p \equiv 3 \pmod{4}$ case, we note that $\lie{g}_{D} = \F_{D} \oplus \F_{D}^{\Tr = 0}$ sitting in degree $1$ and $2$ has Lie bracket given by $[\bar{x},\bar{y}] = \bar{x}\bar{y}^{p} - \bar{y}\bar{x}^{p}$ for any $\bar{x},\bar{y} \in \F_{D}$ in degree $1$ (and $0$ otherwise) by \cite[(6.6)]{Sor}. So, with $\F_{D} = \F_{p}[\alpha]$, we note that
  \[
    [1,\alpha] = 1\cdot\alpha^{p} - \alpha\cdot1^{p} = -2\alpha,
  \]
  since $p \equiv 3 \pmod{4}$ and $\alpha^{2} = -1$. Thus we have an obvious relation between $\lie{g}$ and $\lie{g}_{D}$ given by $\xi_{1},\xi_{3} \leftrightarrow 1,\alpha$ in degree $1$ and $\xi_{2} \leftrightarrow 2\alpha$ in degree $2$. The question is whether we can lift this to a relation between (the representations of) $I$ and $G$. Which we will explore in more detail in \Cref{subsec:quat-algs}.
\end{remark}

\begin{remark}
  In the end of \Cref{rem:quaternion}, we can see that $\lie{g} \iso \lie{g}_{D}$, which could lead one to ask whether a similar fact is true in general for central division algebras $D$ and $I \subseteq \SL_{n}(\Q_{p})$? We check this explicitly for $D$ and $I \subseteq \SL_{n}(\Q_{p})$ with $n=3$ and $p=5$ in \cite{new-code}, and there is no isomorphism between $\lie{g} = \F_{5} \otimes_{\F_{5}[\pi]} \gr I$ and $\lie{g}_{D} = \F_{5} \otimes_{\F_{5}[\pi]} \gr (1+\idm_{D})^{\Nrd=1}$ in this case. This check is done as follows:

  Assume $f \colon \lie{g} \to \lie{g}_{D}$ is a Lie algebra homomorphism, where $x_{1},\dotsc,x_{9}$ is a basis of $\lie{g}$ and $y_{1},\dotsc,y_{9}$ is a basis of $\lie{g}_{D}$, and write $(a_{ij})$ for the matrix representation of $f$ in these bases so that $f(x_{i}) = \sum_{k} a_{ik}y_{k}$. Also write $c_{ijk}^{\lie{g}}$ for the structure constants of $\lie{g}$ and $c_{ijk}^{\lie{g}_{D}}$ for the structure constants of $\lie{g}_{D}$. Then $\bigl[ f(x_{i}),f(x_{j}) \bigr] = f\bigl( [x_{i},x_{j}] \bigr)$ for each $i,j$ implies that \[ \sum_{m,n} a_{im}a_{jn}c_{mnk}^{\lie{g}_{D}} - \sum_{\ell}c_{ij\ell}^{\lie{g}}a_{\ell k} = 0 \] for each $i,j,k$. Writing $e_{ijk}$ for each of the left hand terms above, we can find for $x\in\set{1,2,3,4}$ the ideals $I_{x}$ generated by $\det (a_{ij})_{ij} - x$ and all $e_{ijk}$ (where the $a_{ij}$'s are the variables) and check that $I_{x} = \F_{p}( a_{ij} \mid 1 \leq i,j \leq 9 )$ for each $x=1,2,3,4$. I.e., there is no $(a_{ij})$ satisfying both $c_{ijk} = 0$ for all $i,j,k$ and $\det(a_{ij})_{ij} \neq 0$ in $\F_{5}$, since such $(a_{ij})$ would imply that $\F_{p}( a_{ij} \mid 1 \leq i,j \leq 9 )/I_{x} \neq 0$ for some $x \in \set{ 1,2,3,4 }$, and thus $f$ cannot be an isomorphism. This calculation is all done in \cite{new-code}.
\end{remark}




\section{\texorpdfstring{$H^{*}(I,\F_{p})$}{H*(I,Fp)} for the pro-\texorpdfstring{$p$}{p} Iwahori \texorpdfstring{$I \subseteq \SL_{3}(\Z_{p})$}{I in SL3(Zp)}}%
\label{sec:Iwa-SL3}

In this section we will describe the continuous group cohomology of the pro-$p$ Iwahori subgroup $I$ of $\SL_{3}(\Q_{p})$.

When $I$ is the pro-$p$ Iwahori subgroup in $\SL_{3}(\Q_{p})$, we know by \Cref{subsec:setup-iwa} that we can take it to be of the form
\begin{equation*}
  I = \pmat{1+p\Z_{p} & \Z_{p} & \Z_{p} \\ p\Z_{p} & 1+p\Z_{p} & \Z_{p} \\ p\Z_{p} & p\Z_{p} & 1+p\Z_{p}}^{\!\!\det = 1} \subseteq \SL_{3}(\Z_{p}),
\end{equation*}
and, by \Cref{subsec:setup-iwa}, we have an ordered basis
\begin{equation}
  \label{eq:gis-SL3}
  \begin{gathered}
    g_{1} = \pmat{ 1 \\ & 1 \\ p && 1 }, \quad g_{2} = \pmat{ 1 \\ p & 1 \\ && 1 }, \quad g_{3} = \pmat{ 1 \\ & 1 \\ & p & 1 }, \\
    g_{4} = \pmat{ \exp(p) \\ & \exp(-p) \\ && 1 }, \quad g_{5} = \pmat{ 1 \\ & \exp(p) \\ && \exp(-p) }, \\
    g_{6} = \pmat{ 1 \\ & 1 & 1 \\ && 1 }, \quad g_{7} = \pmat{ 1 & 1 \\ & 1 \\ && 1 }, \quad g_{8} = \pmat{ 1 && 1 \\ & 1 \\ && 1 }.
  \end{gathered}
\end{equation}
Here we write any zeros as blank space in matrices, to make the notation easier to read for the bigger matrices.

By calculations similar to what we did in \Cref{subsec:non-id-xi_ij-SL2}, we get that the non-zero commutators $[\xi_{i},\xi_{j}]$ with $i<j$ are
\begin{equation}
  \label{eq:xi_ij-SL3-main}
  \begin{aligned}
    [\xi_{1},\xi_{6}] &= -\xi_{2}, & [\xi_{1},\xi_{7}] &= \xi_{3}, & [\xi_{1},\xi_{8}] &= -(\xi_{4}+\xi_{5}), \\
    [\xi_{2},\xi_{7}] &= -\xi_{4}, & [\xi_{3},\xi_{6}] &= -\xi_{5}, & [\xi_{6},\xi_{7}] &= -\xi_{8}.
  \end{aligned}
\end{equation}
For the details of these calculations, we refer to \Cref{sec:calc}.

Looking at \eqref{eq:Iwa-p-val-basis-SLn} (with $e=1$ and $h=3$), we see that
\begin{align*}
  \omega(g_{1}) &= 1-\frac{2}{3} = \frac{1}{3}, & \omega(g_{2}) &= 1-\frac{1}{3} = \frac{2}{3}, & \omega(g_{3}) &= 1-\frac{1}{3} = \frac{2}{3}, \\
  \omega(g_{4}) &= 1, & \omega(g_{5}) &= 1, & \omega(g_{6}) &= \frac{1}{3}, \\
  \omega(g_{7}) &= \frac{1}{3}, & \omega(g_{8}) &= \frac{2}{3}.
\end{align*}
Hence
\begin{equation*}
  \lie{g} = \F_{p} \otimes_{\F_{p}[\pi]} \gr I = \Span_{\F_{p}}(\xi_{1},\dotsc,\xi_{8}) = \lie{g}^{1} \oplus \lie{g}^{2} \oplus \lie{g}^{3},
\end{equation*}
where
\begin{align*}
  \lie{g}^{1} &= \lie{g}_{\frac{1}{3}} = \Span_{\F_{p}}(\xi_{1},\xi_{6},\xi_{7}), \\
  \lie{g}^{2} &= \lie{g}_{\frac{2}{3}} = \Span_{\F_{p}}(\xi_{2},\xi_{3},\xi_{8}), \\
  \lie{g}^{3} &= \lie{g}_{1} = \Span_{\F_{p}}(\xi_{4},\xi_{5}).
\end{align*}
See \Cref{rem:g-Z-grading} for more details.

This is all the information we need to calculate the cohomology spaces of $\lie{g}$ (and through that $I$), and using these we can also describe the cup product in the basis we are working with. This time the calculuations are so long that we have moved to doing them by computer (to be precise we are using SageMath). We refer to \cite{new-code} for the details of our implementation, and in \cite{thesis} more details (using a slightly different method) are shown.

We get that
\begin{equation}
  \label{eq:Hn-to-Hst-SL3}
  \begin{aligned}
    H^{0} &= H^{0,0}, \\
    H^{1} &= H^{-1,2}, \\
    H^{2} &= H^{-3,5} \oplus H^{-4,6}, \\
    H^{3} &= H^{-4,7} \oplus H^{-5,8} \oplus H^{-6,9}, \\
    H^{4} &= H^{-7,11} \oplus H^{-8,12}, \\
    H^{5} &= H^{-9,14} \oplus H^{-10,15} \oplus H^{-11,16}, \\
    H^{6} &= H^{-11,17} \oplus H^{-12,18}, \\
    H^{7} &= H^{-14,21}, \\
    H^{8} &= H^{-15,23}
  \end{aligned}
\end{equation}
with dimension as described in \Cref{tab:graded-coh-dims-SL3}.

\begin{table}[ht]
  \centering
  \caption{Dimensions of \texorpdfstring{$E_{1}^{s,t} = H^{s,t} = \gr^{s} H^{s+t}(\lie{g},\F_{p})$}{E1(s,t) = H(s,t)} for the \texorpdfstring{$I \subseteq \SL_{3}(\Z_{p})$}{I in SL3(Zp)} case.}
  \label{tab:graded-coh-dims-SL3}
  \renewcommand{\arraystretch}{1.7}
  \scalebox{0.7}{%
    $\begin{NiceArray}{*{17}{c}}[hvlines, columns-width=auto]
      \diagbox{t}{s} & 0 & -1 & -2 & -3 & -4 & -5 & -6 & -7 & -8 & -9 & -10 & -11 & -12 & -13 & -14 & -15 \\
      0 & 1\\
      1 \\
      2 && 3 \\
      3 \\
      4 \\
      5 &&&& 6\\
      6 &&&&& 3 \\
      7 &&&&& 3\\
      8 &&&&&& 6\\
      9 &&&&&&& 7\\
      10 \\
      11 &&&&&&&& 9 \\
      12 &&&&&&&&& 9 \\
      13 \\
      14 &&&&&&&&&& 7 \\
      15 &&&&&&&&&&& 6 \\
      16 &&&&&&&&&&&& 3 \\
      17 &&&&&&&&&&&& 3 \\
      18 &&&&&&&&&&&&& 6 \\
      19 \\
      20 \\
      21 &&&&&&&&&&&&&&& 3 \\
      22 \\
      23 &&&&&&&&&&&&&&&& 1
    \end{NiceArray}$%
  }
  \renewcommand{\arraystretch}{1}
\end{table}

\subsection{Describing the group cohomology, \texorpdfstring{$H^{n}(I,k)$}{Hn(I,k)}}%
\label{subsec:group-coh-SL3}

We note that all differentials $d_{r}^{s,t} \colon E_{r}^{s,t} \to E_{r}^{s+r,t+1-r}$ in \Cref{tab:graded-coh-dims-SL3} have bidegree $(r,1-r)$, i.e., they are all below the $(r,-r)$ arrow going $r$ to the left and $r$ up in the table, where $r \geq 1$. Looking at \Cref{tab:graded-coh-dims-SL3}, this clearly means that all differentials for $r \geq 1$ are trivial, and thus the spectral sequence collapses on the first page. Hence $H^{s,t}(\lie{g},\F_{p}) = E_{1}^{s,t} \iso E_{\infty}^{s,t} = \gr^{s} H^{s+t}(I,\F_{p})$, and by \eqref{eq:Hn-to-Hst-SL3} and \Cref{tab:graded-coh-dims-SL3} we get that
\begin{equation}
  \label{eq:dim-HnI-SL3}
  \dim H^{n}(I,\F_{p}) =
  \begin{dcases}
    1 & n=0, \\
    3 & n=1, \\
    9 & n=2, \\
    16 & n=3, \\
    18 & n=4, \\
    16 & n=5, \\
    9 & n=6, \\
    3 & n=7, \\
    1 & n=8.
  \end{dcases}
\end{equation}

Recalling that the spectral sequence is multiplicative, we also note, by \Cref{tab:graded-coh-dims-SL3}, that $H^{s,t} \cup H^{s',t'} \subseteq H^{s+s',t+t'}$ implies the we have cup products
\begin{equation*}
  \gr^{s} H^{n}(I,\F_{p}) \otimes \gr^{s'} H^{n'}(I,\F_{p}) \to \gr^{s+s'} H^{n+n'}(I,\F_{p})
\end{equation*}
for $H^{*}(I, \F_{p})$. And, since the spectral sequence collapses on the first page, we can transfer our cup product calculations from $H^{*}(\lie{g},\F_{p})$ to $H^{*}(I,\F_{p})$. We list all the cup products of $H^{s,t}(\lie{g},\F_{p}) \iso \gr^{s} H^{s+t}(I,\F_{p})$ in \Cref{subsec:cup-products}.

\begin{remark}
  We note that \cite{thesis} has a mistake in the part corresponding to the above. In \cite{thesis}, it says that the cup product is trivial, where it should say that the cup product on $H^{*}(I,\F_{p})$ is straightforward to calculate by calculating the cup product on $H^{*}(\lie{g},\F_{p})$ (but it is not trivial).
\end{remark}


\section{Future work}%
\label{sec:future}

In this section we will briefly discuss some interesting future directions of research that were also discussed in \cite{thesis}.

\subsection{Quaternion algebras}%
\label{subsec:quat-algs}

In this subsection, we will further assume that $p>5$ is a prime of the form $p \equiv 3 \pmod{4}$, so that $\Q_{p^{2}} = \Q_{p}(i)$ is the unique unramified quadratic extension of $\Q_{p}$, and $\F_{p^{2}} = \F_{p}[i]$ is the unique quadratic extension of $\F_{p}$.

Let $D$ be the division quaternion algebra over $\Q_{p}$ and let $\tilde{G} = 1+\idm_{D}$ and $G = (1+\idm_{D})^{\Nrd = 1}$, where $\Nrd = \Nrd_{D/\Q_{p}}$ is the norm form. By \cite[Thm.~12.1.5]{Voight} we can assume that $i^{2} = -1$ and $j^{2} = p$ (i.e., we have a tower $D / \Q_{p}(i) / \Q_{p}$), and that $\sO_{D}=\Z_{p}[i,j,k]$ (where $k=ij$) and $\idm_{D} = j\sO_{D} = \sO_{D}j$ (i.e., $\varpi_{D} = j$), which has $\Z_{p}$-basis $p,pi,j,k$, by \cite[Thm.13.1.6]{Voight}.

Now let $\sigma \colon \Q_{p}(i) \to \Q_{p}(i)$ be the complex conjugate and note that $\gen{\sigma} = \Gal\bigl( \Q_{p}(i)/\Q_{p} \bigr)$, so
\begin{equation*}
  D \iso \set[\bigg]{\pmat{ a+bi & c+di \\ p(c-di) & a-bi } \given a,b,c,d\in\Q_{p} } \subseteq M_{2}\bigl( \Q_{p}(i) \bigr)
\end{equation*}
by \cite[Cor.~13.3.14]{Voight}. Hence we have an embedding
\begin{align*}
  D &\hookrightarrow M_{2}\bigl( \Q_{p}(i) \bigr) \\
  a + bi + cj + dk &\mapsto \pmat{ a+bi & c+di \\ p(c-di) & a-bi }
\end{align*}
with
\begin{equation*}
  \Nrd(a+bi+cj+dk) = a^{2} + b^{2} - pc^{2} - pd^{2} = \det\Bigl( \pmat{ a+bi & c+di \\ p(c-di) & a-bi } \Bigr).
\end{equation*}
We note furthermore that $\sO_{D} = \Z_{p}[i] \oplus \Z_{p}[i]j$ and $\idm_{D} = \sO_{D}j = p\Z_{p}[i] \oplus \Z_{p}[i]j$ gives us
\begin{equation*}
  \idm_{D} \iso \set[\bigg]{\pmat{ p(a+bi) & c+di \\ p(c-di) & p(a-bi) } \given a,b,c,d\in\Q_{p} },
\end{equation*}
so $1 + \idm_{D} \subseteq I_{\GL_{2}(\Q_{p}(i))}$, where we denote by $I_{G}$ the (standard choice of) pro-$p$ Iwahori subgroup of $G$ (cf.\ \Cref{subsec:setup-iwa}). Altogether, we get a commutative diagram
\begin{equation}\label{eq:quat-diagram}
  \begin{tikzcd}
    (1+\idm_{D})^{\Nrd = 1} \ar[d, hook] \ar[r, hook] & I_{\SL_{2}(\Q_{p}(i))} \ar[d, hook] & \ar[l, hook'] I_{\SL_{2}(\Q_{p})} \ar[d, hook] \\
    (\sO_{D}^{\times})^{\Nrd = 1} \ar[d, hook] \ar[r, hook] & \SL_{2}\bigl( \Z_{p}[i] \bigr) \ar[d, hook] & \ar[l, hook'] \SL_{2}(\Z_{p}) \ar[d, hook] \\
    (D^{\times})^{\Nrd = 1} \ar[d, hook] \ar[r, hook] & \SL_{2}\bigl( \Q_{p}(i) \bigr) \ar[d, hook] & \ar[l, hook'] \SL_{2}(\Q_{p}) \ar[d, hook] \\
    D^{\times} \ar[r, hook] & \GL_{2}\bigl( \Q_{p}(i) \bigr) & \ar[l, hook'] \GL_{2}(\Q_{p}) \\
    \sO_{D}^{\times} \ar[u, hook] \ar[r, hook] & \GL_{2}\bigl( \Z_{p}[i] \bigr) \ar[u, hook] & \ar[l, hook'] \GL_{2}(\Z_{p}) \ar[u, hook] \\
    1+\idm_{D} \ar[u, hook] \ar[r, hook] & I_{\GL_{2}(\Q_{p}(i))} \ar[u, hook] & \ar[l, hook'] I_{\GL_{2}(\Q_{p})}. \ar[u, hook]
  \end{tikzcd}
\end{equation}

We saw in \Cref{rem:quaternion} that $H^{*}(G,\F_{p}) \iso H^{*}(I_{\SL_{2}(\Q_{p})},\F_{p})$, and in \Cref{prop:GLn-connection} we noted that $H^{*}(I_{\SL_{2}(\Q_{p})},\F_{p}) \otimes_{F_{p}} \F_{p}[\varepsilon] \iso H^{*}(I_{\GL_{2}(\Q_{p})},\F_{p})$ (where $\varepsilon^{2} = 0$) and $H^{*}(\tilde{G},\F_{p}) \iso H^{*}(G,\F_{p}) \otimes_{\F_{p}} \F_{p}[\varepsilon]$, so $H^{*}(\tilde{G},\F_{p}) \iso H^{*}(I_{\GL_{2}(\Q_{p})},\F_{p})$. (Recall that $G = (1+\idm_{D})^{\Nrd = 1}$ and $\tilde{G} = 1+\idm_{D}$.) Furthermore, \cite[Sect.~6.3]{Sor} argues that $H^{*}(\sO_{D}^{\times},\F_{p}) \iso H^{*}(\tilde{G},\F_{p})^{\F_{D}^{\times}}$, using that we can factor $\sO_{D}^{\times}$ as a semi-direct product $\tilde{G} \rtimes \F_{D}^{\times}$. Here the $\F_{D}^{\times}$-action on $H^{*}(\tilde{G},\F_{p})$ is understood and non-trivial, cf.\ \cite[Prop.~7~(b)]{Henn}. An interesting question is, if the comparison between cohomology of the left side and right side of \eqref{eq:quat-diagram} can be somehow continued?

\begin{remark}
  To see that $H^{*}(\tilde{G},\F_{p}) \iso H^{*}(I_{\GL_{2}(\Q_{p})},\F_{p})$ for $p\geq5$ in general, and not just for $p \equiv 3 \pmod{4}$, one can compare the basis and structure of $H^{*}(\tilde{G},\F_{p})$ described in \cite[Thm.~3.2]{Ravenel} with the basis and structure described in \cite[Sect.~3.4.3]{thesis}.
\end{remark}

Another interesting direction of research is to note that we already have bijections between \emph{certain} mod $p$ representations of $D^{\times}$ and $\GL_{2}(\Q_{p})$ from the Jacquet-Langlands correspondence for $\GL_{2}$ (cf.\ \cite{JL}), and we can ask whether there are similar relations in between the left and right side of the other rows of \eqref{eq:quat-diagram}. Here we note that by \cite[Rem.~4.5]{JL-remark} irreducible representations of $D^{\times}$ are trivial on $1+\idm_{D}$, so we need something new if we want a correspondence between certain mod $p$ representations of $G = (1+\idm_{D})^{\Nrd = 1}$ and $I_{\SL_{2}(\Q_{p})}$ or between certain mod $p$ representations of $\tilde{G} = 1+\idm_{D}$ and $I_{\GL_{2}(\Q_{p})}$.

Finally, although we already have isomorphisms $H^{*}(G,\F_{p}) \iso H^{*}(I_{\SL_{2}(\Q_{p})},\F_{p})$ and $H^{*}(\tilde{G},\F_{p}) \iso H^{*}(I_{\GL_{2}(\Q_{p})},\F_{p})$, we note that these were obtained by concrete calculations, and we would really prefer to have canonical isomorphisms (possibly obtained by working with the corresponding row of \eqref{eq:quat-diagram}).

In pursuit of the canonical isomorphisms mentioned above, we note that one can show by explicit calculations (with bases) that the inclusions of \eqref{eq:quat-diagram} give inclusions
\begin{equation*}
  \begin{tikzcd}
    \gr (1+\idm_{D})^{\Nrd = 1} \ar[r, hook] & \gr I_{\SL_{2}(\Q_{p}(i))} & \ar[l, hook'] \gr I_{\SL_{2}(\Q_{p})}, \\
    \gr (1+\idm_{D}) \ar[r, hook] & \gr I_{\GL_{2}(\Q_{p}(i))} & \ar[l, hook'] \gr I_{\GL_{2}(\Q_{p})},
  \end{tikzcd}
\end{equation*}
where the pro-$p$ Iwahori subgroups are graded as usual (start with $\gr I = \bigoplus_{\nu>0} \gr_{\nu} I$ where $\gr_{\nu} I = I_{\nu}/I_{\nu+}$ and translate to $\gr^{i} I$), $1+\idm_{D}$ is graded by $\gr^{i}(1+\idm_{D}) = (1+\idm_{D}^{i})/(1+\idm_{D}^{i+1})$, and $(1+\idm_{D})^{\Nrd = 1}$ is graded by $\gr^{i}(1+\idm_{D})^{\Nrd = 1} = (1+\idm_{D}^{i})^{\Nrd = 1}/(1+\idm_{D}^{i+1})^{\Nrd = 1}$. These inclusions further translate to inclusions
\begin{equation*}
  \begin{tikzcd}
    \lie{g}_{(1+\idm_{D})^{\Nrd = 1}} \ar[r, hook] & \lie{g}_{I_{\SL_{2}(\Q_{p}(i))}} & \ar[l, hook'] \lie{g}_{I_{\SL_{2}(\Q_{p})}}, \\
    \lie{g}_{(1+\idm_{D})} \ar[r, hook] & \lie{g}_{I_{\GL_{2}(\Q_{p}(i))}} & \ar[l, hook'] \lie{g}_{I_{\GL_{2}(\Q_{p})}},
  \end{tikzcd}
\end{equation*}
where $\lie{g}_{G} = \F_{p} \otimes_{\F_{p}[\pi]} \gr G$ is the Lazard Lie algebra of $G$. We note that these inclusions do not have the same images, but we noted in \Cref{rem:quaternion} that $\lie{g}_{(1+\idm_{D})^{\Nrd = 1}} \iso \lie{g}_{I_{\SL_{2}(\Q_{p}(i))}}$, so we might be able to come up with a canonical isomorphism through these somehow.


\subsection{Central division algebras}%
\label{subsec:central-div-algs}

Let $D$ be the central division algebra over $\Q_{p}$ of dimension $n^{2}$ and invariant $\frac{1}{n}$. Recall the following setup from \cite[Sect.~6.3]{Sor}: The valuation $v_{p}$ on $\Q_{p}$ extends uniquely to a valuation $\tilde{v} \colon D^{\times} \to \frac{1}{n}\Z$ by the formula $\tilde{v}(x) = \frac{1}{n}v\bigl(\Nrd_{D/\Q_{p}}(x)\bigr)$, and the valuation ring $\sO_{D} = \set{ x : \tilde{v}(x) >0 }$ is the maximal compact subring of $D$. It is local with maximal ideal $\idm_{D} = \set{ x : \tilde{v}(x) > 0 }$ and residue field $\F_{D} \iso \F_{p^{n}}$. Furthermore, we can pick $\varpi_{D}$ such that $\tilde{v}(\varpi_{D}) = \frac{1}{n}$, $\idm_{D} = \varpi_{D}\sO_{D} = \sO_{D}\varpi_{D}$ and $p = \varpi_{D}^{n}$.

When $p > n+1$, we also saw in \cite[Sect.~6.3]{Sor} that $\tilde{G} = 1 + \idm_{D}$ has Lazard Lie algebra $\tilde{\lie{g}} = \F_{D} \oplus \dotsb \oplus \F_{D}$ concentrated in degrees $1,2,\dotsc,n$ with Lie bracket given by the formula
\begin{equation}\label{eq:central-div-alg-com}
  [x,y] = xy^{p^{i}} - yx^{p^{j}}
\end{equation}
for $x \in \tilde{\lie{g}}^{i} \iso \F_{D}$ and $y \in \tilde{\lie{g}}^{j}$. Furthermore $G = (1 + \idm_{D})^{\Nrd = 1}$ has Lazard Lie algebra $\lie{g} = \F_{D} \oplus \dotsb \oplus \F_{D} \oplus \F_{D}^{\Tr = 0}$ concentrated in degrees $1,2,\dotsc,n$ with Lie bracket given by \eqref{eq:central-div-alg-com}. (Note that one can easily check that $[x,y]$ has trace zero when $i+j = n$.) Here $\F_{D}^{\Tr = 0}$ is the kernel of the trace $\Tr_{\F_{D}/\F_{p}}$ and $\lie{g} \subseteq \tilde{\lie{g}}$ is a codimension one Lie subalgebra.

In the previous subsection we focused on the case $n = 2$ (and $p \equiv 3 \pmod{4}$), but one can ask if some of the ideas work in more general cases. For the remainder of this subsection we will focus on the case $n = 3$ and $p = 5$.

We note that $x^{3} + 3x + 3$ is an irreducible polynomial in $\F_{5}[x]$, so $\F_{D} \iso \F_{5^{3}} \iso \F_{5}[\alpha]$ where $\alpha^{3} = -3\alpha-3 = 2\alpha + 2$. Now let $\xi_{1} = 1, \xi_{2} = \alpha, \xi_{3} = \alpha^{2}$ be the basis of $\tilde{\lie{g}}^{1} \iso \F_{D}$, let $\xi_{4} = 1, \xi_{5} = \alpha, \xi_{6} = \alpha^{2}$ be the basis of $\tilde{\lie{g}}^{2} \iso \F_{D}$, and let $\xi_{7} = 1, \xi_{8} = \alpha, \xi_{9} = \alpha^{2}$ be the basis of $\tilde{\lie{g}}^{3} \iso \F_{D}$. Using \eqref{eq:central-div-alg-com}, we see that
\begin{equation}\label{eq:central-div-alg-first-com}
  \begin{aligned}
    [\xi_{1},\xi_{2}] &= 4\xi_{4} + 3\xi_{5} + 2\xi_{6}, & [\xi_{1},\xi_{3}] &= 3\xi_{4} + 2\xi_{5} + 4\xi_{6}, \\
    [\xi_{1},\xi_{5}] &= 4\xi_{7} + 3\xi_{8} + 2\xi_{9}, & [\xi_{1},\xi_{6}] &= 3\xi_{7} + 2\xi_{8} + 4\xi_{9}, \\
    [\xi_{2},\xi_{3}] &= 2\xi_{4} + \xi_{5} + 4\xi_{6}, & [\xi_{2},\xi_{4}] &= 4\xi_{7} + \xi_{8} + 2\xi_{9}, \\
    [\xi_{2},\xi_{5}] &= 3\xi_{7} + \xi_{8} + 4\xi_{9}, & [\xi_{2},\xi_{6}] &= 2\xi_{8}, \\
    [\xi_{3},\xi_{4}] &= 4\xi_{7} + 2\xi_{8} + 2\xi_{9}, & [\xi_{3},\xi_{5}] &= 3\xi_{8}, \\
    [\xi_{3},\xi_{6}] &= 3\xi_{7} + 4\xi_{9}.
  \end{aligned}
\end{equation}
Here
\begin{equation*}
  \tilde{\lie{g}}^{1} = \Span_{\F_{5}}(\xi_{1},\xi_{2},\xi_{3}), \quad \tilde{\lie{g}}^{2} = \Span_{\F_{5}}(\xi_{4},\xi_{5},\xi_{6}), \quad \tilde{\lie{g}}^{3} = \Span_{\F_{5}}(\xi_{7},\xi_{8},\xi_{9}),
\end{equation*}
so we order the basis by the index of the $\xi_{i}$'s.

Continuing, we can recall that $\Tr_{\F_{D}/\F_{5}}(x) = x + x^{5} + x^{5^{2}}$ for $x \in \F_{D} \iso \F_{5^{3}}$, so
\begin{align*}
  \Tr_{\F_{D}/\F_{5}}(1) &= 3, & \Tr_{\F_{D}/\F_{5}}(\alpha) &= 0, & \Tr_{\F_{D}/\F_{5}}(\alpha^{2}) &= 4,
\end{align*}
since $\alpha^{3} = 2\alpha + 2$. Thus $\F_{D}^{\Tr = 0}$ has basis $\alpha, 4+2\alpha^{2}$. Now let $\xi_{1}' = 1, \xi_{2}' = \alpha, \xi_{3}' = \alpha^{2}$ be the basis of $\lie{g}^{1} \iso \F_{D}$, let $\xi_{4}' = 1, \xi_{5}' = \alpha, \xi_{6}' = \alpha^{2}$ be the basis of $\lie{g}^{2} \iso \F_{D}$, and let $\xi_{7}' = \alpha, \xi_{8}' = 4+2\alpha$ be the basis of $\lie{g}^{3} \iso \F_{D}^{\Tr = 0}$. Using \eqref{eq:central-div-alg-com}, we see that
\begin{equation}\label{eq:central-div-alg-second-com}
  \begin{aligned}
    [\xi_{1}',\xi_{2}'] &= 4\xi_{4}' + 3\xi_{5}' + 2\xi_{6}', & [\xi_{1}',\xi_{3}'] &= 3\xi_{4}' + 2\xi_{5}' + 4\xi_{6}', \\
    [\xi_{1}',\xi_{5}'] &= 3\xi_{7}' + \xi_{8}', & [\xi_{1}',\xi_{6}'] &= 2\xi_{7}' + 2\xi_{8}', \\
    [\xi_{2}',\xi_{3}'] &= 2\xi_{4}' + \xi_{5}' + 4\xi_{6}', & [\xi_{2}',\xi_{4}'] &= \xi_{7}' + \xi_{8}', \\
    [\xi_{2}',\xi_{5}'] &= \xi_{7}' + 2\xi_{8}', & [\xi_{2}',\xi_{6}'] &= 2\xi_{7}', \\
    [\xi_{3}',\xi_{4}'] &= 2\xi_{7}' + \xi_{8}', & [\xi_{3}',\xi_{5}'] &= 3\xi_{7}', \\
    [\xi_{3}',\xi_{6}'] &= 2\xi_{8}'.
  \end{aligned}
\end{equation}
Here
\begin{equation*}
  \lie{g}^{1} = \Span_{\F_{5}}(\xi_{1}',\xi_{2}',\xi_{3}'), \quad \lie{g}^{2} = \Span_{\F_{5}}(\xi_{4}',\xi_{5}',\xi_{6}'), \quad \lie{g}^{3} = \Span_{\F_{5}}(\xi_{7}',\xi_{8}'),
\end{equation*}
so we order the basis by the index of the $\xi_{i}'$'s. Calculating the cohomology as in the previous sections with this information, we get \Cref{tab:graded-coh-dims-central-div-alg-3-prime}

\begin{remark}
  Note that when calculating the cohomology here, we need to do all calculations modulo $5$ since \eqref{eq:central-div-alg-second-com} do not lift to a Lie algebra over $\Z$ with these Chevalley constants. See \cite{code} for the details.
\end{remark}

\begin{table}[ht]
  \centering
  \caption{Dimensions of \texorpdfstring{$E_{1}^{s,t} = H^{s,t} = \gr^{s} H^{s+t}(\lie{g},\F_{5})$}{E1(s,t) = H(s,t)} for \texorpdfstring{$G = (1+\idm_{D})^{\Nrd = 1}$}{G = (1+mD) with Nrd = 1} in the \texorpdfstring{$n=3$}{n = 3} and \texorpdfstring{$p=5$}{p = 5} case.}
  \label{tab:graded-coh-dims-central-div-alg-3-prime}
  \renewcommand{\arraystretch}{1.7}
  \scalebox{0.7}{%
    $\begin{NiceArray}{*{17}{c}}[hvlines, columns-width=auto]
      \diagbox{t}{s} & 0 & -1 & -2 & -3 & -4 & -5 & -6 & -7 & -8 & -9 & -10 & -11 & -12 & -13 & -14 & -15 \\
      0 & 1\\
      1 \\
      2 && 3 \\
      3 \\
      4 \\
      5 &&&& 6\\
      6 &&&&& 3 \\
      7 &&&&& 3\\
      8 &&&&&& 6\\
      9 &&&&&&& 7\\
      10 \\
      11 &&&&&&&& 9 \\
      12 &&&&&&&&& 9 \\
      13 \\
      14 &&&&&&&&&& 7 \\
      15 &&&&&&&&&&& 6 \\
      16 &&&&&&&&&&&& 3 \\
      17 &&&&&&&&&&&& 3 \\
      18 &&&&&&&&&&&&& 6 \\
      19 \\
      20 \\
      21 &&&&&&&&&&&&&&& 3 \\
      22 \\
      23 &&&&&&&&&&&&&&&& 1
    \end{NiceArray}$%
  }
  \renewcommand{\arraystretch}{1}
\end{table}

Comparing \Cref{tab:graded-coh-dims-SL3} and \Cref{tab:graded-coh-dims-central-div-alg-3-prime}, we see that $H^{*}(I,\F_{5})$ for $I \subseteq \SL_{3}(\Z_{5})$ and $H^{*}\bigl((1+\idm_{D})^{\Nrd = 1},\F_{5}\bigr)$ have the same graded cohomology dimensions, and it would be interesting to investigate whether $H^{*}(I,\F_{5}) \iso H^{*}\bigl((1+\idm_{D})^{\Nrd = 1},\F_{5}\bigr)$ as graded algebras. More generally, is $H^{*}(I,\F_{p}) \iso H^{*}\bigl((1+\idm_{D})^{\Nrd = 1},\F_{p}\bigr)$ as graded algebras for $p \geq 5$?

As in the $I \subseteq \SL_{3}(\F_{p})$ case, we can do all calculations by computer and find all the (graded) cup products of $H^{*}\bigl( (1+\idm_{D})^{\Nrd=1},\F_{5} )$. We have chosen not to include these in this paper, but they can be found in \cite{new-code}, and we note that the cup products are so complicated that it is hard to tell whether we can have an isomorphism between $H^{*}\bigl( I,\F_{5} )$ and $H^{*}\bigl( (1+\idm_{D})^{\Nrd=1},\F_{5} )$.

Altogether, the above seems to hint at the following conjecture:

\begin{conjecture}
  Let $D$ be the central division algebra over $\Q_{p}$ of dimension $n^{2}$ and invariant $\frac{1}{n}$. Let $\sO_{D}$ be the maximal compact (local) subring of $D$ with maximal ideal $\idm_{D}$ and residue field $\F_{D} \iso \F_{p^{n}}$. If $p > n+1$ then
  \begin{enumerate}[$\bullet$]
    \item $H^{*}\bigl( I_{\GL_{n}(\Q_{p})}, \F_{p} \bigr) \iso H^{*}\bigl( 1+\idm_{D}, \F_{p} \bigr)$ as graded algebras, and
    \item $H^{*}\bigl( I_{\SL_{n}(\Q_{p})}, \F_{p} \bigr) \iso H^{*}\bigl( (1+\idm_{D})^{\Nrd = 1}, \F_{p} \bigr)$ as graded algebras.
  \end{enumerate}
\end{conjecture}

We note that the first part of the conjecture will follow from the second part by \Cref{prop:GLn-connection}.

\subsection{Serre spectral sequence}%
\label{subsec:Serre-spec-seq}

Another interesting research direction is to try to work with the Serre spectral sequence in the following way.

Assume we have the \enquote{standard} setup with $\gs{G} = \SL_{n}$, $\gs{U}$ unipotent upper triangular matrices and $\gs{T}$ diagonal matrices with determinant $1$. Let also $I \subseteq \SL_{n}(\Z_{p})$ be the pro-$p$ Iwahori subgroup of $\SL_{n}(\Q_{p})$ which is upper triangular and unipotent modulo $p$, and let
\begin{equation*}
  K \coloneqq \kernel\bigl( \red \colon \gs{G}(\Z_{p}) \to \gs{G}(\F_{p}) \bigr),
\end{equation*}
where $\red \colon \gs{G}(\Z_{p}) \to \gs{G}(\F_{p})$ is the reduction map. (Note that $I = \set{g \in \gs{G}(\Z_{p}) : \red(g) \in \gs{U}(\F_{p})}$ in this case, cf.\ \cite{Generators}.) Then
\begin{equation*}
  I/K \iso \gs{U}(\F_{p}),
\end{equation*}
and thus we get the Serre spectral sequence
\begin{equation*}
  E_{2}^{i,j} = H^{i}\bigl( \gs{U}(\F_{p}), H^{j}(K,\F_{p}) \bigr) \Longrightarrow H^{i+j}(I,\F_{p}),
\end{equation*}
which is also a multiplicative spectral sequence. Since $K$ is a uniformly powerful group (cf.\ \cite[Prop.~7.6]{SchOll-modular}), we know by \cite[p.~183]{Laz} that
\begin{equation*}
  H^{j}(K,\F_{p}) \iso \bigwedge^{j} \Hom_{\F_{p}}(K,\F_{p}).
\end{equation*}
Now we can let $\SL_{n}(\Z_{p})$ act by
\begin{equation*}
  (g \act f)(x) = f(g^{-1}xg)
\end{equation*}
for $g \in \SL_{n}(\Z_{p})$, $f \colon K \to \F_{p}$ and $x \in K$, and hope to split $\bigwedge^{j} \Hom_{\F_{p}}(K,\F_{p})$ into a direct sum of Verma modules $\bigoplus_{\lambda} V(\lambda)$ for $p$-restricted $\lambda$ ($\lambda$ with $0 \leq \inner{\lambda}{\alpha^{\vee}} \leq p-1$), similarly to what is done in \cite{PT}. This description might be easier to compare with a similar spectral sequence for $(1+\idm_{D})^{\Nrd=1}$, but it is harder to get started with since the spectral sequence is more complicated. One can hope that the difference in the spectral sequence might make it so that it will always collapse on the second page (the starting page in this case). We note that the author has recently been attempting to use this approach, but the fact that even $H^{*}(\gs{U}(\F_{p}), \F_{p})$ is not very well understood for $n>3$ makes this approach hard to generalize. The hope is that we can connect it directly with a spectral sequence for $(1+\idm_{D})^{\Nrd=1}$ without discussing where the spectral sequence collapses.


\printbibliography

\newpage

\appendix
\renewcommand{\theHsection}{A\arabic{section}}
\markboth{\scshape APPENDIX}{\scshape APPENDIX}

\section{Calculations for the \texorpdfstring{$I \subseteq \SL_{3}(\Q_{p})$}{Sl3(Qp)} case}%
\label{sec:calc}

In this appendix we will show some of the calculations for finding the cohomology of the pro-$p$ Iwahori subgroup $I$ of $\SL_{3}(\Q_{p})$.

\subsection{Finding the commutators \texorpdfstring{$[\xi_{i},\xi_{j}]$}{[xi-i,xi-j]}}%
\label{subsec:non-id-xi_ij-SL3}

Using the basis of \eqref{eq:gis-SL3} from \Cref{sec:Iwa-SL3}, we note that
\begin{equation*}
    g_{1}^{x_{1}}g_{2}^{x_{2}}g_{3}^{x_{3}}g_{4}^{x_{4}}g_{5}^{x_{5}}g_{6}^{x_{6}}g_{7}^{x_{7}}g_{8}^{x_{8}} = \pmat{ a_{11} & a_{12} & a_{13} \\ a_{21} & a_{22} & a_{23} \\ a_{31} & a_{32} & a_{33}},
\end{equation*}
where
\begin{equation}
  \label{eq:gixi-SL3}
  \begin{aligned}
    a_{11} &= \exp(px_{4}), \\
    a_{12} &= x_{7}\exp(px_{4}), \\
    a_{13} &= x_{8}\exp(px_{4}), \\
    a_{21} &= px_{2}\exp(px_{4}), \\
    a_{22} &= px_{2}x_{7}\exp(px_{4}) + \exp\bigl( p(x_{5}-x_{4}) \bigr), \\
    a_{23} &= px_{2}x_{8}\exp(px_{4}) + x_{6}\exp\bigl( p(x_{5}-x_{4}) \bigr), \\
    a_{31} &= px_{1}\exp(px_{4}), \\
    a_{32} &= px_{1}x_{7}\exp(px_{4}) + px_{3}\exp\bigl( p(x_{5}-x_{4}) \bigr), \\
    a_{33} &= px_{1}x_{8}\exp(px_{4}) + px_{3}x_{6}\exp\bigl( p(x_{5}-x_{4}) \bigr) + \exp(-px_{5}).
  \end{aligned}
\end{equation}

Writing $g_{ij} = [g_{i},g_{j}]$ and $\xi_{ij} = [\xi_{i},\xi_{j}]$, we are now ready to find $x_{1},\dotsc,x_{8}$ such that $g_{ij} = g_{1}^{x_{1}} \dotsb g_{8}^{x_{8}}$ for different $i<j$. (In the following we use that $\frac{1}{p-1} = 1 + p + p^{2} + \dotsb$ and $\log(1-p) = -p - \frac{p^{2}}{2} - \frac{p^{3}}{3} - \dotsb$.) Also, except in the first case, we will note that $x_{k} \in p\Z_{p}$ implies that the coefficient on $\xi_{k}$ in $\xi_{ij}$ is zero.

We now list all non-identity commutators $g_{ij} = [g_{i},g_{j}]$ and find $\xi_{ij} = [\xi_{i},\xi_{j}]$ based on these. (For $g_{ij} = 1_{3}$ it is clear that $x_{1} = \cdots = x_{8} = 0$, and thus $\xi_{ij} = 0$.)

\begin{description}
  \item[$g_{14} = \pmat{1 \\ & 1 \\ p\bigl( 1-\exp(-p) \bigr) && 1}$] Comparing $g_{14}$ with \eqref{eq:gixi-SL3}, we see that $x_{2} = x_{4} = x_{7} = x_{8} = 0$, and thus also $x_{3} = x_{5} = x_{6} = 0$. This leaves $a_{31} = px_{1} = p\bigl( 1-\exp(-p) \bigr) = p^{2} + O(p^{3})$, which implies that $x_{1} = p + O(p^{2})$. Hence $\sigma(g_{14}) = \pi \act \sigma(g_{1})$, which implies that $\xi_{14} = 0$.

  \item[$g_{15} = \pmat{1 \\ & 1 \\ p\bigl( 1-\exp(-p) \bigr) && 1}$] Since $g_{15} = g_{14}$, the above calculation shows that $\xi_{15} = 0$.

  \item[$g_{16} = \pmat{1 \\ -p & 1 \\ && 1}$] Comparing $g_{16}$ with \eqref{eq:gixi-SL3}, we see that $x_{1} = x_{4} = x_{7} = x_{8} = 0$, and thus also $x_{3} = x_{5} = x_{6} = 0$. This leaves $a_{21} = px_{2} = -p$, which implies that $x_{2} = -1$. Hence $\sigma(g_{16}) = -\sigma(g_{2})$, which implies that $\xi_{16} = -\xi_{2}$.

  \item[$g_{17} = \pmat{1 \\ & 1 \\ & p & 1}$] Comparing $g_{17}$ with \eqref{eq:gixi-SL3}, we see that $x_{1} = x_{2} = x_{4} = x_{7} = x_{8} = 0$, and thus also $x_{5} = x_{6} = 0$. This leaves $a_{32} = px_{3} = p$, which implies that $x_{3} = 1$. Hence $\sigma(g_{17}) = \sigma(g_{3})$, which implies that $\xi_{17} = \xi_{3}$.

  \item[$g_{18} = \pmat{1-p && p \\ & 1 \\ -p^{2} && 1+p+p^{2}}$] Comparing $g_{18}$ with \eqref{eq:gixi-SL3}, we see that $x_{2} = x_{7} = 0$, and thus also $x_{3} = x_{6} = 0$ and $x_{4} = x_{5}$. Using
        \begin{align*}
          a_{11} &= \exp(px_{4}) = 1-p, \\
          a_{13} &= x_{8}\exp(px_{4}) = x_{8}(1-p) = p, \\
          a_{31} &= px_{1}\exp(px_{4}) = px_{1}(1-p) = -p^{2},
        \end{align*}
        we get that
        \begin{align*}
          x_{4} &= \dfrac{1}{p}\log(1-p) = \dfrac{1}{p}\bigl( (-p) + O(p^{2}) \bigr) = -1 + O(p), \\
          x_{8} &= \dfrac{p}{1-p} = p + O(p^{2}), \\
          x_{1} &= \dfrac{-p^{2}}{p(1-p)} = -p + O(p^{2}).
        \end{align*}
        Hence $\sigma(g_{18}) = -\pi \act \sigma(g_{1}) - \sigma(g_{4}) - \sigma(g_{5}) + \pi \act \sigma(g_{8})$, which implies that $\xi_{18} = -(\xi_{4}+\xi_{5})$.

  \item[$g_{23} = \pmat{1 \\ & 1 \\ -p^{2} && 1}$] Comparing $g_{23}$ with \eqref{eq:gixi-SL3}, we see that $x_{2} = x_{4} = x_{7} = x_{8} = 0$, and thus also $x_{3} = x_{5} = x_{6} = 0$. This leaves $a_{31} = px_{1} = -p^{2}$, which implies that $x_{1} = -p$. Hence $\sigma(g_{23}) = -\pi \act \sigma(g_{1})$, which implies that $\xi_{23} = 0$.

  \item[$g_{24} = \pmat{1 \\ p\bigl( 1-\exp(-2p) \bigr) & 1 \\ && 1}$] Comparing $g_{24}$ with \eqref{eq:gixi-SL3}, we see that $x_{1} = x_{4} = x_{7} = x_{8} = 0$, and thus also $x_{3} = x_{5} = x_{6} = 0$. This leaves $a_{21} = px_{2} = p\bigl( 1-\exp(-2p) \bigr) = p\bigl( 1-\bigl( 1+(-2p)+O(p^{2}) \bigr) \bigr) = 2p^{2} + O(p^{3})$, which implies that $x_{2} = 2p + O(p^{2})$. Hence $\sigma(g_{24}) = 2\pi \act \sigma(g_{1})$, which implies that $\xi_{24} = 0$.

  \item[$g_{25} = \pmat{1 \\ p\bigl( 1-\exp(p) \bigr) & 1 \\ && 1}$] Except a factor $-2$ in the exponential, which clearly does not change the final result, we have the same calculation as for $g_{24}$. Thus $\xi_{25} = 0$.

  \item[$g_{27} = \pmat{ 1-p & p \\ -p^{2} & 1+p+p^{2} \\ && 1}$] Comparing $g_{27}$ with \eqref{eq:gixi-SL3}, we see that $x_{1} = x_{8} = 0$, and thus also $x_{3} = x_{6} = 0$, so $x_{5} = 0$. Using
        \begin{align*}
          a_{11} &= \exp(px_{4}) = 1-p, \\
          a_{12} &= x_{7}\exp(px_{4}) = x_{8}(1-p) = p, \\
          a_{21} &= px_{2}\exp(px_{4}) = px_{2}(1-p) = -p^{2},
        \end{align*}
        we get that
        \begin{align*}
          x_{4} &= \dfrac{1}{p}\log(1-p) = \dfrac{1}{p}\bigl( (-p) + O(p^{2}) \bigr) = -1 + O(p), \\
          x_{7} &= \dfrac{p}{1-p} = p + O(p^{2}), \\
          x_{2} &= \dfrac{-p^{2}}{p(1-p)} = -p + O(p^{2}).
        \end{align*}
        Hence $\sigma(g_{27}) = -\pi \act \sigma(g_{2}) - \sigma(g_{4}) + \pi \act \sigma(g_{7})$, which implies that $\xi_{27} = -\xi_{4}$.

  \item[$g_{28} = \pmat{ 1 \\ & 1 & p \\ && 1}$] Comparing $g_{28}$ with \eqref{eq:gixi-SL3}, we see that $x_{1} = x_{2} = x_{4} = x_{7} = x_{8} = 0$, and thus also $x_{3} = x_{5} = 0$. This leaves $a_{23} = x_{6} = p$. Hence $\sigma(g_{28}) = \pi \act \sigma(g_{6})$, which implies that $\xi_{28} = 0$.

  \item[$g_{34} = \pmat{ 1 \\ & 1 \\ & p\bigl( 1-\exp(p) \bigr) & 1}$] Comparing $g_{34}$ with \eqref{eq:gixi-SL3}, we see that $x_{1} = x_{2} = x_{4} = x_{7} = x_{8} = 0$, and thus also $x_{5} = x_{6} = 0$. This leaves $a_{32} = px_{3} = p\bigl( 1-\exp(p) \bigr) = p\bigl( 1-\bigl( 1+p+O(p^{2}) \bigr) \bigr) = -p^{2} + O(p^{3})$, which implies that $x_{3} = -p + O(p^{2})$. Hence $\sigma(g_{34}) = -\pi \act \sigma(g_{3})$, which implies that $\xi_{34} = 0$.

  \item[$g_{35} = \pmat{ 1 \\ & 1 \\ & p\bigl( 1-\exp(-2p) \bigr) & 1}$] Except a factor $-2$ in the exponential, which clearly does not change the final result, we have the same calculation as for $g_{34}$. Thus $\xi_{35} = 0$.

  \item[$g_{36} = \pmat{ 1 \\ & 1-p & p \\ & -p^{2} & 1+p+p^{2}}$] Comparing $g_{36}$ with \eqref{eq:gixi-SL3}, we see that $x_{1} = x_{2} = x_{4} = x_{7} = x_{8} = 0$. Using
        \begin{align*}
          a_{22} &= \exp(px_{5}) = 1-p, \\
          a_{23} &= x_{6}\exp(px_{5}) = x_{6}(1-p) = p, \\
          a_{32} &= px_{3}\exp(px_{5}) = px_{3}(1-p) = -p^{2},
        \end{align*}
        we get that
        \begin{align*}
          x_{5} &= \dfrac{1}{p}\log(1-p) = \dfrac{1}{p}\bigl( (-p) + O(p^{2}) \bigr) = -1 + O(p), \\
          x_{6} &= \dfrac{p}{1-p} = p + O(p^{2}), \\
          x_{3} &= \dfrac{-p^{2}}{p(1-p)} = -p + O(p^{2}).
        \end{align*}
        Hence $\sigma(g_{36}) = -\pi \act \sigma(g_{3}) - \sigma(g_{5}) + \pi \act \sigma(g_{6})$, which implies that $\xi_{36} = -\xi_{5}$.

  \item[$g_{38} = \pmat{ 1 & -p \\ & 1 \\ && 1}$] Comparing $g_{38}$ with \eqref{eq:gixi-SL3}, we see that $x_{1} = x_{2} = x_{4} = x_{8} = 0$, and thus also $x_{3} = x_{5} = x_{6} = 0$. This leaves $a_{12} = x_{7} = -p$. Hence $\sigma(g_{38}) = -\pi \act \sigma(g_{3})$, which implies that $\xi_{38} = 0$.

  \item[$g_{46} = \pmat{ 1 \\ & 1 & \exp(-p)-1 \\ && 1}$] Comparing $g_{46}$ with \eqref{eq:gixi-SL3}, we see that $x_{1} = x_{2} = x_{4} = x_{7} = x_{8} = 0$, and thus also $x_{3} = x_{5} = 0$. This leaves $a_{23} = x_{6} = \exp(-p) - 1 = -p + O(p^{2})$. Hence $\sigma(g_{46}) = -\pi \act \sigma(g_{6})$, which implies that $\xi_{46} = 0$.

  \item[$g_{47} = \pmat{ 1 & \exp(2p)-1 \\ & 1 \\ && 1}$] Comparing $g_{47}$ with \eqref{eq:gixi-SL3}, we see that $x_{1} = x_{2} = x_{4} = x_{8} = 0$, and thus also $x_{3} = x_{5} = x_{6} = 0$. This leaves $a_{12} = x_{7} = \exp(2p) - 1 = 2p + O(p^{2})$. Hence $\sigma(g_{47}) = 2\pi \act \sigma(g_{7})$, which implies that $\xi_{47} = 0$.

  \item[$g_{48} = \pmat{ 1 && \exp(p)-1 \\ & 1 \\ && 1}$] Comparing $g_{48}$ with \eqref{eq:gixi-SL3}, we see that $x_{1} = x_{2} = x_{4} = x_{7} = 0$, and thus also $x_{3} = x_{5} = x_{6} = 0$. This leaves $a_{13} = x_{8} = \exp(p) - 1 = p + O(p^{2})$. Hence $\sigma(g_{48}) = \pi \act \sigma(g_{8})$, which implies that $\xi_{48} = 0$.

  \item[$g_{56} = \pmat{ 1 \\ & 1 & \exp(2p)-1 \\ && 1}$] Except a factor $-2$ in the exponential, which clearly does not change the final result, we have the same calculation as for $g_{46}$. Thus $\xi_{56} = 0$.

  \item[$g_{57} = \pmat{ 1 & \exp(-p)-1 \\ & 1 \\ && 1}$] Except a factor $-2$ in the exponential, which clearly does not change the final result, we have the same calculation as for $g_{47}$. Thus $\xi_{57} = 0$.

  \item[$g_{58} = \pmat{ 1 && \exp(p)-1 \\ & 1 \\ && 1}$] Since $g_{58} = g_{48}$, the above calculation shows that $\xi_{58} = 0$.

  \item[$g_{67} = \pmat{ 1 && -1 \\ & 1 \\ && 1}$] Comparing $g_{67}$ with \eqref{eq:gixi-SL3}, we see that $x_{1} = x_{2} = x_{4} = x_{7} = 0$, and thus also $x_{3} = x_{5} = x_{6} = 0$. This leaves $a_{13} = x_{8} = -1$. Hence $\sigma(g_{67}) = -\sigma(g_{8})$, which implies that $\xi_{67} = -\xi_{8}$.
\end{description}

Thus the non-zero commutators $[\xi_{i},\xi_{j}]$ with $i<j$ are:
\begin{equation}
  \label{eq:xi_ij-SL3}
  \begin{aligned}
    [\xi_{1},\xi_{6}] &= -\xi_{2}, & [\xi_{1},\xi_{7}] &= \xi_{3}, & [\xi_{1},\xi_{8}] &= -(\xi_{4}+\xi_{5}), \\
    [\xi_{2},\xi_{7}] &= -\xi_{4}, & [\xi_{3},\xi_{6}] &= -\xi_{5}, & [\xi_{6},\xi_{7}] &= -\xi_{8}.
  \end{aligned}
\end{equation}

Looking at \eqref{eq:Iwa-p-val-basis-SLn} (with $e=1$ and $h=3$), we see that
\begin{align*}
  \omega(g_{1}) &= 1-\frac{2}{3} = \frac{1}{3}, & \omega(g_{2}) &= 1-\frac{1}{3} = \frac{2}{3}, & \omega(g_{3}) &= 1-\frac{1}{3} = \frac{2}{3}, \\
  \omega(g_{4}) &= 1, & \omega(g_{5}) &= 1, & \omega(g_{6}) &= \frac{1}{3}, \\
  \omega(g_{7}) &= \frac{1}{3}, & \omega(g_{8}) &= \frac{2}{3}.
\end{align*}
Hence
\begin{equation*}
  \lie{g} = \F_{p} \otimes_{\F_{p}[\pi]} \gr I = \Span_{\F_{p}}(\xi_{1},\dotsc,\xi_{8}) = \lie{g}^{1} \oplus \lie{g}^{2} \oplus \lie{g}^{3},
\end{equation*}
where
\begin{align*}
  \lie{g}^{1} &= \lie{g}_{\frac{1}{3}} = \Span_{\F_{p}}(\xi_{1},\xi_{6},\xi_{7}), \\
  \lie{g}^{2} &= \lie{g}_{\frac{2}{3}} = \Span_{\F_{p}}(\xi_{2},\xi_{3},\xi_{8}), \\
  \lie{g}^{3} &= \lie{g}_{1} = \Span_{\F_{p}}(\xi_{4},\xi_{5}).
\end{align*}
See \Cref{rem:g-Z-grading} for more details.

\subsection{Cup products}%
\label{subsec:cup-products}

For each non-zero $H^{s,t} = H^{s,t}(\lie{g},\F_{p})$, write $(v_{1}^{(s,t)},\dotsc,v_{m}^{(s,t)})$ for the basis of $H^{s,t}$ (with $m = \dim_{\F_{p}} H^{s,t}$). Then our calculations in \cite{new-code} show that we can describe all non-zero cup products by the following formulas.

\vspace*{1em}

\begin{multicols}{2}

\begin{center}
  \scalebox{1.2}{$H^{-1,2} \cup H^{-3,5} \subseteq H^{-4,7}$}
\end{center}

\begin{align*}
v_{1}^{(-1,2)} \cup v_{3}^{(-3,5)} &= v_{1}^{(-4,7)} \\
v_{1}^{(-1,2)} \cup v_{5}^{(-3,5)} &= v_{2}^{(-4,7)} \\
v_{2}^{(-1,2)} \cup v_{1}^{(-3,5)} &= -v_{1}^{(-4,7)} \\
v_{2}^{(-1,2)} \cup v_{6}^{(-3,5)} &= v_{3}^{(-4,7)} \\
v_{3}^{(-1,2)} \cup v_{2}^{(-3,5)} &= -v_{2}^{(-4,7)} \\
v_{3}^{(-1,2)} \cup v_{4}^{(-3,5)} &= -v_{3}^{(-4,7)}
\end{align*}
\vspace*{1em}

\begin{center}
  \scalebox{1.2}{$H^{-1,2} \cup H^{-4,6} \subseteq H^{-5,8}$}
\end{center}

\begin{align*}
v_{1}^{(-1,2)} \cup v_{1}^{(-4,6)} &= -v_{4}^{(-5,8)} \\
v_{1}^{(-1,2)} \cup v_{2}^{(-4,6)} &= -2v_{1}^{(-5,8)} \\
v_{1}^{(-1,2)} \cup v_{3}^{(-4,6)} &= 2v_{2}^{(-5,8)} \\
v_{2}^{(-1,2)} \cup v_{1}^{(-4,6)} &= -2v_{1}^{(-5,8)} \\
v_{2}^{(-1,2)} \cup v_{2}^{(-4,6)} &= v_{5}^{(-5,8)} \\
v_{2}^{(-1,2)} \cup v_{3}^{(-4,6)} &= -2v_{3}^{(-5,8)} \\
v_{3}^{(-1,2)} \cup v_{1}^{(-4,6)} &= -2v_{2}^{(-5,8)} \\
v_{3}^{(-1,2)} \cup v_{2}^{(-4,6)} &= 2v_{3}^{(-5,8)} \\
v_{3}^{(-1,2)} \cup v_{3}^{(-4,6)} &= v_{6}^{(-5,8)}
\end{align*}
\vspace*{1em}

\begin{center}
  \scalebox{1.2}{$H^{-1,2} \cup H^{-6,9} \subseteq H^{-7,11}$}
\end{center}

\begin{align*}
v_{1}^{(-1,2)} \cup v_{3}^{(-6,9)} &= 3v_{3}^{(-7,11)} \\
v_{1}^{(-1,2)} \cup v_{4}^{(-6,9)} &= v_{1}^{(-7,11)} - v_{2}^{(-7,11)} \\
v_{1}^{(-1,2)} \cup v_{5}^{(-6,9)} &= -v_{4}^{(-7,11)} \\
v_{1}^{(-1,2)} \cup v_{6}^{(-6,9)} &= v_{5}^{(-7,11)} - v_{6}^{(-7,11)} \\
v_{1}^{(-1,2)} \cup v_{7}^{(-6,9)} &= v_{7}^{(-7,11)} \\
v_{2}^{(-1,2)} \cup v_{1}^{(-6,9)} &= -v_{1}^{(-7,11)} \\
v_{2}^{(-1,2)} \cup v_{2}^{(-6,9)} &= v_{3}^{(-7,11)} \\
v_{2}^{(-1,2)} \cup v_{3}^{(-6,9)} &= 3v_{4}^{(-7,11)} \\
v_{2}^{(-1,2)} \cup v_{6}^{(-6,9)} &= v_{7}^{(-7,11)} \\
v_{2}^{(-1,2)} \cup v_{7}^{(-6,9)} &= v_{8}^{(-7,11)} \\
v_{3}^{(-1,2)} \cup v_{1}^{(-6,9)} &= v_{3}^{(-7,11)} \\
v_{3}^{(-1,2)} \cup v_{2}^{(-6,9)} &= -v_{6}^{(-7,11)} \\
v_{3}^{(-1,2)} \cup v_{3}^{(-6,9)} &= -3v_{7}^{(-7,11)} \\
v_{3}^{(-1,2)} \cup v_{4}^{(-6,9)} &= v_{4}^{(-7,11)} \\
v_{3}^{(-1,2)} \cup v_{5}^{(-6,9)} &= -v_{9}^{(-7,11)}
\end{align*}
\vspace*{1em}

\begin{center}
  \scalebox{1.2}{$H^{-1,2} \cup H^{-8,12} \subseteq H^{-9,14}$}
\end{center}

\begin{align*}
v_{1}^{(-1,2)} \cup v_{1}^{(-8,12)} &= v_{2}^{(-9,14)} \\
v_{1}^{(-1,2)} \cup v_{2}^{(-8,12)} &= -v_{4}^{(-9,14)} \\
v_{1}^{(-1,2)} \cup v_{3}^{(-8,12)} &= 3v_{1}^{(-9,14)} \\
v_{1}^{(-1,2)} \cup v_{6}^{(-8,12)} &= v_{3}^{(-9,14)} \\
v_{1}^{(-1,2)} \cup v_{7}^{(-8,12)} &= v_{3}^{(-9,14)} \\
v_{1}^{(-1,2)} \cup v_{8}^{(-8,12)} &= v_{5}^{(-9,14)} \\
v_{1}^{(-1,2)} \cup v_{9}^{(-8,12)} &= v_{5}^{(-9,14)} \\
v_{2}^{(-1,2)} \cup v_{1}^{(-8,12)} &= -v_{3}^{(-9,14)} \\
v_{2}^{(-1,2)} \cup v_{2}^{(-8,12)} &= -3v_{1}^{(-9,14)} \\
v_{2}^{(-1,2)} \cup v_{3}^{(-8,12)} &= -v_{6}^{(-9,14)} \\
v_{2}^{(-1,2)} \cup v_{5}^{(-8,12)} &= -v_{2}^{(-9,14)} \\
v_{2}^{(-1,2)} \cup v_{9}^{(-8,12)} &= v_{7}^{(-9,14)} \\
v_{3}^{(-1,2)} \cup v_{1}^{(-8,12)} &= 3v_{1}^{(-9,14)} \\
v_{3}^{(-1,2)} \cup v_{2}^{(-8,12)} &= -v_{5}^{(-9,14)} \\
v_{3}^{(-1,2)} \cup v_{3}^{(-8,12)} &= -v_{7}^{(-9,14)} \\
v_{3}^{(-1,2)} \cup v_{4}^{(-8,12)} &= -v_{4}^{(-9,14)} \\
v_{3}^{(-1,2)} \cup v_{6}^{(-8,12)} &= -v_{6}^{(-9,14)}
\end{align*}
\vspace*{1em}

\begin{center}
  \scalebox{1.2}{$H^{-1,2} \cup H^{-10,15} \subseteq H^{-11,17}$}
\end{center}

\begin{align*}
v_{1}^{(-1,2)} \cup v_{3}^{(-10,15)} &= 2v_{1}^{(-11,17)} \\
v_{1}^{(-1,2)} \cup v_{5}^{(-10,15)} &= -2v_{2}^{(-11,17)} \\
v_{1}^{(-1,2)} \cup v_{6}^{(-10,15)} &= v_{3}^{(-11,17)} \\
v_{2}^{(-1,2)} \cup v_{2}^{(-10,15)} &= 2v_{1}^{(-11,17)} \\
v_{2}^{(-1,2)} \cup v_{4}^{(-10,15)} &= -v_{2}^{(-11,17)} \\
v_{2}^{(-1,2)} \cup v_{5}^{(-10,15)} &= -2v_{3}^{(-11,17)} \\
v_{3}^{(-1,2)} \cup v_{1}^{(-10,15)} &= v_{1}^{(-11,17)} \\
v_{3}^{(-1,2)} \cup v_{2}^{(-10,15)} &= 2v_{2}^{(-11,17)} \\
v_{3}^{(-1,2)} \cup v_{3}^{(-10,15)} &= 2v_{3}^{(-11,17)}
\end{align*}
\vspace*{1em}

\begin{center}
  \scalebox{1.2}{$H^{-1,2} \cup H^{-11,16} \subseteq H^{-12,18}$}
\end{center}

\begin{align*}
v_{1}^{(-1,2)} \cup v_{2}^{(-11,16)} &= v_{2}^{(-12,18)} \\
v_{1}^{(-1,2)} \cup v_{3}^{(-11,16)} &= v_{4}^{(-12,18)} \\
v_{2}^{(-1,2)} \cup v_{1}^{(-11,16)} &= -v_{1}^{(-12,18)} \\
v_{2}^{(-1,2)} \cup v_{3}^{(-11,16)} &= v_{6}^{(-12,18)} \\
v_{3}^{(-1,2)} \cup v_{1}^{(-11,16)} &= -v_{3}^{(-12,18)} \\
v_{3}^{(-1,2)} \cup v_{2}^{(-11,16)} &= -v_{5}^{(-12,18)}
\end{align*}
\vspace*{1em}

\begin{center}
  \scalebox{1.2}{$H^{-1,2} \cup H^{-14,21} \subseteq H^{-15,23}$}
\end{center}

\begin{align*}
v_{1}^{(-1,2)} \cup v_{3}^{(-14,21)} &= v_{1}^{(-15,23)} \\
v_{2}^{(-1,2)} \cup v_{2}^{(-14,21)} &= -v_{1}^{(-15,23)} \\
v_{3}^{(-1,2)} \cup v_{1}^{(-14,21)} &= v_{1}^{(-15,23)}
\end{align*}
\vspace*{1em}

\begin{center}
  \scalebox{1.2}{$H^{-3,5} \cup H^{-4,6} \subseteq H^{-7,11}$}
\end{center}

\begin{align*}
v_{1}^{(-3,5)} \cup v_{2}^{(-4,6)} &= -v_{1}^{(-7,11)} + 2v_{2}^{(-7,11)} \\
v_{1}^{(-3,5)} \cup v_{3}^{(-4,6)} &= 3v_{3}^{(-7,11)} \\
v_{2}^{(-3,5)} \cup v_{2}^{(-4,6)} &= -3v_{3}^{(-7,11)} \\
v_{2}^{(-3,5)} \cup v_{3}^{(-4,6)} &= -2v_{5}^{(-7,11)} + v_{6}^{(-7,11)} \\
v_{3}^{(-3,5)} \cup v_{1}^{(-4,6)} &= v_{1}^{(-7,11)} + v_{2}^{(-7,11)} \\
v_{3}^{(-3,5)} \cup v_{3}^{(-4,6)} &= 3v_{4}^{(-7,11)} \\
v_{4}^{(-3,5)} \cup v_{1}^{(-4,6)} &= -3v_{4}^{(-7,11)} \\
v_{4}^{(-3,5)} \cup v_{3}^{(-4,6)} &= -2v_{8}^{(-7,11)} + v_{9}^{(-7,11)} \\
v_{5}^{(-3,5)} \cup v_{1}^{(-4,6)} &= v_{5}^{(-7,11)} + v_{6}^{(-7,11)} \\
v_{5}^{(-3,5)} \cup v_{2}^{(-4,6)} &= 3v_{7}^{(-7,11)} \\
v_{6}^{(-3,5)} \cup v_{1}^{(-4,6)} &= 3v_{7}^{(-7,11)} \\
v_{6}^{(-3,5)} \cup v_{2}^{(-4,6)} &= v_{8}^{(-7,11)} - 2v_{9}^{(-7,11)}
\end{align*}
\vspace*{1em}

\begin{center}
  \scalebox{1.2}{$H^{-3,5} \cup H^{-6,9} \subseteq H^{-9,14}$}
\end{center}

\begin{align*}
v_{1}^{(-3,5)} \cup v_{3}^{(-6,9)} &= v_{2}^{(-9,14)} \\
v_{1}^{(-3,5)} \cup v_{5}^{(-6,9)} &= -v_{3}^{(-9,14)} \\
v_{1}^{(-3,5)} \cup v_{6}^{(-6,9)} &= -v_{4}^{(-9,14)} \\
v_{1}^{(-3,5)} \cup v_{7}^{(-6,9)} &= v_{1}^{(-9,14)} \\
v_{2}^{(-3,5)} \cup v_{3}^{(-6,9)} &= -v_{4}^{(-9,14)} \\
v_{2}^{(-3,5)} \cup v_{4}^{(-6,9)} &= -v_{2}^{(-9,14)} \\
v_{2}^{(-3,5)} \cup v_{5}^{(-6,9)} &= -v_{1}^{(-9,14)} \\
v_{2}^{(-3,5)} \cup v_{7}^{(-6,9)} &= -v_{5}^{(-9,14)} \\
v_{3}^{(-3,5)} \cup v_{2}^{(-6,9)} &= v_{2}^{(-9,14)} \\
v_{3}^{(-3,5)} \cup v_{3}^{(-6,9)} &= v_{3}^{(-9,14)} \\
v_{3}^{(-3,5)} \cup v_{6}^{(-6,9)} &= -v_{1}^{(-9,14)} \\
v_{3}^{(-3,5)} \cup v_{7}^{(-6,9)} &= -v_{6}^{(-9,14)} \\
v_{4}^{(-3,5)} \cup v_{1}^{(-6,9)} &= -v_{3}^{(-9,14)} \\
v_{4}^{(-3,5)} \cup v_{2}^{(-6,9)} &= -v_{1}^{(-9,14)} \\
v_{4}^{(-3,5)} \cup v_{3}^{(-6,9)} &= -v_{6}^{(-9,14)} \\
v_{4}^{(-3,5)} \cup v_{6}^{(-6,9)} &= -v_{7}^{(-9,14)} \\
v_{5}^{(-3,5)} \cup v_{1}^{(-6,9)} &= -v_{4}^{(-9,14)} \\
v_{5}^{(-3,5)} \cup v_{3}^{(-6,9)} &= v_{5}^{(-9,14)} \\
v_{5}^{(-3,5)} \cup v_{4}^{(-6,9)} &= -v_{1}^{(-9,14)} \\
v_{5}^{(-3,5)} \cup v_{5}^{(-6,9)} &= v_{7}^{(-9,14)} \\
v_{6}^{(-3,5)} \cup v_{1}^{(-6,9)} &= v_{1}^{(-9,14)} \\
v_{6}^{(-3,5)} \cup v_{2}^{(-6,9)} &= -v_{5}^{(-9,14)} \\
v_{6}^{(-3,5)} \cup v_{3}^{(-6,9)} &= v_{7}^{(-9,14)} \\
v_{6}^{(-3,5)} \cup v_{4}^{(-6,9)} &= -v_{6}^{(-9,14)}
\end{align*}
\vspace*{1em}

\begin{center}
  \scalebox{1.2}{$H^{-3,5} \cup H^{-8,12} \subseteq H^{-11,17}$}
\end{center}

\begin{align*}
v_{1}^{(-3,5)} \cup v_{3}^{(-8,12)} &= 3v_{1}^{(-11,17)} \\
v_{1}^{(-3,5)} \cup v_{8}^{(-8,12)} &= 2v_{2}^{(-11,17)} \\
v_{1}^{(-3,5)} \cup v_{9}^{(-8,12)} &= v_{2}^{(-11,17)} \\
v_{2}^{(-3,5)} \cup v_{3}^{(-8,12)} &= 3v_{2}^{(-11,17)} \\
v_{2}^{(-3,5)} \cup v_{6}^{(-8,12)} &= -v_{1}^{(-11,17)} \\
v_{2}^{(-3,5)} \cup v_{7}^{(-8,12)} &= -2v_{1}^{(-11,17)} \\
v_{3}^{(-3,5)} \cup v_{2}^{(-8,12)} &= -3v_{1}^{(-11,17)} \\
v_{3}^{(-3,5)} \cup v_{8}^{(-8,12)} &= v_{3}^{(-11,17)} \\
v_{3}^{(-3,5)} \cup v_{9}^{(-8,12)} &= -v_{3}^{(-11,17)} \\
v_{4}^{(-3,5)} \cup v_{2}^{(-8,12)} &= -3v_{3}^{(-11,17)} \\
v_{4}^{(-3,5)} \cup v_{4}^{(-8,12)} &= -v_{1}^{(-11,17)} \\
v_{4}^{(-3,5)} \cup v_{5}^{(-8,12)} &= -2v_{1}^{(-11,17)} \\
v_{5}^{(-3,5)} \cup v_{1}^{(-8,12)} &= 3v_{2}^{(-11,17)} \\
v_{5}^{(-3,5)} \cup v_{6}^{(-8,12)} &= v_{3}^{(-11,17)} \\
v_{5}^{(-3,5)} \cup v_{7}^{(-8,12)} &= -v_{3}^{(-11,17)} \\
v_{6}^{(-3,5)} \cup v_{1}^{(-8,12)} &= 3v_{3}^{(-11,17)} \\
v_{6}^{(-3,5)} \cup v_{4}^{(-8,12)} &= -2v_{2}^{(-11,17)} \\
v_{6}^{(-3,5)} \cup v_{5}^{(-8,12)} &= -v_{2}^{(-11,17)}
\end{align*}
\vspace*{1em}

\begin{center}
  \scalebox{1.2}{$H^{-3,5} \cup H^{-11,16} \subseteq H^{-14,21}$}
\end{center}

\begin{align*}
v_{1}^{(-3,5)} \cup v_{3}^{(-11,16)} &= -v_{2}^{(-14,21)} \\
v_{2}^{(-3,5)} \cup v_{2}^{(-11,16)} &= v_{1}^{(-14,21)} \\
v_{3}^{(-3,5)} \cup v_{3}^{(-11,16)} &= -v_{3}^{(-14,21)} \\
v_{4}^{(-3,5)} \cup v_{1}^{(-11,16)} &= v_{1}^{(-14,21)} \\
v_{5}^{(-3,5)} \cup v_{2}^{(-11,16)} &= -v_{3}^{(-14,21)} \\
v_{6}^{(-3,5)} \cup v_{1}^{(-11,16)} &= v_{2}^{(-14,21)}
\end{align*}
\vspace*{1em}

\begin{center}
  \scalebox{1.2}{$H^{-3,5} \cup H^{-12,18} \subseteq H^{-15,23}$}
\end{center}

\begin{align*}
v_{1}^{(-3,5)} \cup v_{6}^{(-12,18)} &= v_{1}^{(-15,23)} \\
v_{2}^{(-3,5)} \cup v_{5}^{(-12,18)} &= -v_{1}^{(-15,23)} \\
v_{3}^{(-3,5)} \cup v_{4}^{(-12,18)} &= -v_{1}^{(-15,23)} \\
v_{4}^{(-3,5)} \cup v_{3}^{(-12,18)} &= -v_{1}^{(-15,23)} \\
v_{5}^{(-3,5)} \cup v_{2}^{(-12,18)} &= -v_{1}^{(-15,23)} \\
v_{6}^{(-3,5)} \cup v_{1}^{(-12,18)} &= v_{1}^{(-15,23)}
\end{align*}
\vspace*{1em}

\begin{center}
  \scalebox{1.2}{$H^{-4,6} \cup H^{-4,6} \subseteq H^{-8,12}$}
\end{center}

\begin{align*}
v_{1}^{(-4,6)} \cup v_{1}^{(-4,6)} &= -2v_{4}^{(-8,12)} - 2v_{5}^{(-8,12)} \\
v_{1}^{(-4,6)} \cup v_{2}^{(-4,6)} &= 2v_{1}^{(-8,12)} \\
v_{1}^{(-4,6)} \cup v_{3}^{(-4,6)} &= 2v_{2}^{(-8,12)} \\
v_{2}^{(-4,6)} \cup v_{1}^{(-4,6)} &= 2v_{1}^{(-8,12)} \\
v_{2}^{(-4,6)} \cup v_{2}^{(-4,6)} &= 2v_{6}^{(-8,12)} - 4v_{7}^{(-8,12)} \\
v_{2}^{(-4,6)} \cup v_{3}^{(-4,6)} &= -2v_{3}^{(-8,12)} \\
v_{3}^{(-4,6)} \cup v_{1}^{(-4,6)} &= 2v_{2}^{(-8,12)} \\
v_{3}^{(-4,6)} \cup v_{2}^{(-4,6)} &= -2v_{3}^{(-8,12)} \\
v_{3}^{(-4,6)} \cup v_{3}^{(-4,6)} &= 4v_{8}^{(-8,12)} - 2v_{9}^{(-8,12)}
\end{align*}
\vspace*{1em}

\begin{center}
  \scalebox{1.2}{$H^{-4,6} \cup H^{-5,8} \subseteq H^{-9,14}$}
\end{center}

\begin{align*}
v_{1}^{(-4,6)} \cup v_{1}^{(-5,8)} &= -v_{2}^{(-9,14)} \\
v_{1}^{(-4,6)} \cup v_{2}^{(-5,8)} &= -v_{4}^{(-9,14)} \\
v_{1}^{(-4,6)} \cup v_{3}^{(-5,8)} &= 3v_{1}^{(-9,14)} \\
v_{1}^{(-4,6)} \cup v_{5}^{(-5,8)} &= -2v_{3}^{(-9,14)} \\
v_{1}^{(-4,6)} \cup v_{6}^{(-5,8)} &= -2v_{5}^{(-9,14)} \\
v_{2}^{(-4,6)} \cup v_{1}^{(-5,8)} &= v_{3}^{(-9,14)} \\
v_{2}^{(-4,6)} \cup v_{2}^{(-5,8)} &= -3v_{1}^{(-9,14)} \\
v_{2}^{(-4,6)} \cup v_{3}^{(-5,8)} &= -v_{6}^{(-9,14)} \\
v_{2}^{(-4,6)} \cup v_{4}^{(-5,8)} &= -2v_{2}^{(-9,14)} \\
v_{2}^{(-4,6)} \cup v_{6}^{(-5,8)} &= 2v_{7}^{(-9,14)} \\
v_{3}^{(-4,6)} \cup v_{1}^{(-5,8)} &= 3v_{1}^{(-9,14)} \\
v_{3}^{(-4,6)} \cup v_{2}^{(-5,8)} &= v_{5}^{(-9,14)} \\
v_{3}^{(-4,6)} \cup v_{3}^{(-5,8)} &= v_{7}^{(-9,14)} \\
v_{3}^{(-4,6)} \cup v_{4}^{(-5,8)} &= 2v_{4}^{(-9,14)} \\
v_{3}^{(-4,6)} \cup v_{5}^{(-5,8)} &= 2v_{6}^{(-9,14)}
\end{align*}
\vspace*{1em}

\begin{center}
  \scalebox{1.2}{$H^{-4,6} \cup H^{-6,9} \subseteq H^{-10,15}$}
\end{center}

\begin{align*}
v_{1}^{(-4,6)} \cup v_{3}^{(-6,9)} &= -3v_{2}^{(-10,15)} \\
v_{1}^{(-4,6)} \cup v_{4}^{(-6,9)} &= -2v_{1}^{(-10,15)} \\
v_{1}^{(-4,6)} \cup v_{5}^{(-6,9)} &= v_{3}^{(-10,15)} \\
v_{1}^{(-4,6)} \cup v_{6}^{(-6,9)} &= -2v_{4}^{(-10,15)} \\
v_{1}^{(-4,6)} \cup v_{7}^{(-6,9)} &= -v_{5}^{(-10,15)} \\
v_{2}^{(-4,6)} \cup v_{1}^{(-6,9)} &= -2v_{1}^{(-10,15)} \\
v_{2}^{(-4,6)} \cup v_{2}^{(-6,9)} &= -v_{2}^{(-10,15)} \\
v_{2}^{(-4,6)} \cup v_{3}^{(-6,9)} &= -3v_{3}^{(-10,15)} \\
v_{2}^{(-4,6)} \cup v_{6}^{(-6,9)} &= -v_{5}^{(-10,15)} \\
v_{2}^{(-4,6)} \cup v_{7}^{(-6,9)} &= 2v_{6}^{(-10,15)} \\
v_{3}^{(-4,6)} \cup v_{1}^{(-6,9)} &= v_{2}^{(-10,15)} \\
v_{3}^{(-4,6)} \cup v_{2}^{(-6,9)} &= -2v_{4}^{(-10,15)} \\
v_{3}^{(-4,6)} \cup v_{3}^{(-6,9)} &= -3v_{5}^{(-10,15)} \\
v_{3}^{(-4,6)} \cup v_{4}^{(-6,9)} &= v_{3}^{(-10,15)} \\
v_{3}^{(-4,6)} \cup v_{5}^{(-6,9)} &= -2v_{6}^{(-10,15)}
\end{align*}
\vspace*{1em}

\begin{center}
  \scalebox{1.2}{$H^{-4,6} \cup H^{-7,11} \subseteq H^{-11,17}$}
\end{center}

\begin{align*}
v_{1}^{(-4,6)} \cup v_{4}^{(-7,11)} &= -2v_{1}^{(-11,17)} \\
v_{1}^{(-4,6)} \cup v_{7}^{(-7,11)} &= 2v_{2}^{(-11,17)} \\
v_{1}^{(-4,6)} \cup v_{8}^{(-7,11)} &= 2v_{3}^{(-11,17)} \\
v_{1}^{(-4,6)} \cup v_{9}^{(-7,11)} &= -2v_{3}^{(-11,17)} \\
v_{2}^{(-4,6)} \cup v_{3}^{(-7,11)} &= -2v_{1}^{(-11,17)} \\
v_{2}^{(-4,6)} \cup v_{5}^{(-7,11)} &= 4v_{2}^{(-11,17)} \\
v_{2}^{(-4,6)} \cup v_{6}^{(-7,11)} &= 2v_{2}^{(-11,17)} \\
v_{2}^{(-4,6)} \cup v_{7}^{(-7,11)} &= 2v_{3}^{(-11,17)} \\
v_{3}^{(-4,6)} \cup v_{1}^{(-7,11)} &= -2v_{1}^{(-11,17)} \\
v_{3}^{(-4,6)} \cup v_{2}^{(-7,11)} &= -4v_{1}^{(-11,17)} \\
v_{3}^{(-4,6)} \cup v_{3}^{(-7,11)} &= 2v_{2}^{(-11,17)} \\
v_{3}^{(-4,6)} \cup v_{4}^{(-7,11)} &= 2v_{3}^{(-11,17)}
\end{align*}
\vspace*{1em}

\begin{center}
  \scalebox{1.2}{$H^{-4,6} \cup H^{-8,12} \subseteq H^{-12,18}$}
\end{center}

\begin{align*}
v_{1}^{(-4,6)} \cup v_{1}^{(-8,12)} &= -3v_{1}^{(-12,18)} \\
v_{1}^{(-4,6)} \cup v_{2}^{(-8,12)} &= -3v_{3}^{(-12,18)} \\
v_{1}^{(-4,6)} \cup v_{6}^{(-8,12)} &= v_{2}^{(-12,18)} \\
v_{1}^{(-4,6)} \cup v_{7}^{(-8,12)} &= -v_{2}^{(-12,18)} \\
v_{1}^{(-4,6)} \cup v_{8}^{(-8,12)} &= v_{4}^{(-12,18)} \\
v_{1}^{(-4,6)} \cup v_{9}^{(-8,12)} &= -v_{4}^{(-12,18)} \\
v_{2}^{(-4,6)} \cup v_{1}^{(-8,12)} &= 3v_{2}^{(-12,18)} \\
v_{2}^{(-4,6)} \cup v_{3}^{(-8,12)} &= 3v_{5}^{(-12,18)} \\
v_{2}^{(-4,6)} \cup v_{4}^{(-8,12)} &= 2v_{1}^{(-12,18)} \\
v_{2}^{(-4,6)} \cup v_{5}^{(-8,12)} &= v_{1}^{(-12,18)} \\
v_{2}^{(-4,6)} \cup v_{8}^{(-8,12)} &= -2v_{6}^{(-12,18)} \\
v_{2}^{(-4,6)} \cup v_{9}^{(-8,12)} &= -v_{6}^{(-12,18)} \\
v_{3}^{(-4,6)} \cup v_{2}^{(-8,12)} &= 3v_{4}^{(-12,18)} \\
v_{3}^{(-4,6)} \cup v_{3}^{(-8,12)} &= 3v_{6}^{(-12,18)} \\
v_{3}^{(-4,6)} \cup v_{4}^{(-8,12)} &= v_{3}^{(-12,18)} \\
v_{3}^{(-4,6)} \cup v_{5}^{(-8,12)} &= 2v_{3}^{(-12,18)} \\
v_{3}^{(-4,6)} \cup v_{6}^{(-8,12)} &= v_{5}^{(-12,18)} \\
v_{3}^{(-4,6)} \cup v_{7}^{(-8,12)} &= 2v_{5}^{(-12,18)}
\end{align*}
\vspace*{1em}

\begin{center}
  \scalebox{1.2}{$H^{-4,6} \cup H^{-10,15} \subseteq H^{-15,23}$}
\end{center}

\begin{align*}
v_{1}^{(-4,6)} \cup v_{3}^{(-10,15)} &= -2v_{1}^{(-14,21)} \\
v_{1}^{(-4,6)} \cup v_{5}^{(-10,15)} &= -2v_{2}^{(-14,21)} \\
v_{1}^{(-4,6)} \cup v_{6}^{(-10,15)} &= -v_{3}^{(-14,21)} \\
v_{2}^{(-4,6)} \cup v_{2}^{(-10,15)} &= -2v_{1}^{(-14,21)} \\
v_{2}^{(-4,6)} \cup v_{4}^{(-10,15)} &= -v_{2}^{(-14,21)} \\
v_{2}^{(-4,6)} \cup v_{5}^{(-10,15)} &= 2v_{3}^{(-14,21)} \\
v_{3}^{(-4,6)} \cup v_{1}^{(-10,15)} &= v_{1}^{(-14,21)} \\
v_{3}^{(-4,6)} \cup v_{2}^{(-10,15)} &= -2v_{2}^{(-14,21)} \\
v_{3}^{(-4,6)} \cup v_{3}^{(-10,15)} &= 2v_{3}^{(-14,21)}
\end{align*}
\vspace*{1em}

\begin{center}
  \scalebox{1.2}{$H^{-4,6} \cup H^{-11,17} \subseteq H^{-15,23}$}
\end{center}

\begin{align*}
v_{1}^{(-4,6)} \cup v_{3}^{(-11,17)} &= -v_{1}^{(-15,23)} \\
v_{2}^{(-4,6)} \cup v_{2}^{(-11,17)} &= -v_{1}^{(-15,23)} \\
v_{3}^{(-4,6)} \cup v_{1}^{(-11,17)} &= v_{1}^{(-15,23)}
\end{align*}
\vspace*{1em}

\begin{center}
  \scalebox{1.2}{$H^{-4,7} \cup H^{-11,16} \subseteq H^{-15,23}$}
\end{center}

\begin{align*}
v_{1}^{(-4,7)} \cup v_{3}^{(-11,16)} &= -v_{1}^{(-15,23)} \\
v_{2}^{(-4,7)} \cup v_{2}^{(-11,16)} &= -v_{1}^{(-15,23)} \\
v_{3}^{(-4,7)} \cup v_{1}^{(-11,16)} &= -v_{1}^{(-15,23)}
\end{align*}
\vspace*{1em}

\begin{center}
  \scalebox{1.2}{$H^{-5,8} \cup H^{-6,9} \subseteq H^{-11,17}$}
\end{center}

\begin{align*}
v_{1}^{(-5,8)} \cup v_{3}^{(-6,9)} &= 3v_{1}^{(-11,17)} \\
v_{1}^{(-5,8)} \cup v_{6}^{(-6,9)} &= -v_{2}^{(-11,17)} \\
v_{1}^{(-5,8)} \cup v_{7}^{(-6,9)} &= -v_{3}^{(-11,17)} \\
v_{2}^{(-5,8)} \cup v_{3}^{(-6,9)} &= 3v_{2}^{(-11,17)} \\
v_{2}^{(-5,8)} \cup v_{4}^{(-6,9)} &= v_{1}^{(-11,17)} \\
v_{2}^{(-5,8)} \cup v_{5}^{(-6,9)} &= -v_{3}^{(-11,17)} \\
v_{3}^{(-5,8)} \cup v_{1}^{(-6,9)} &= -v_{1}^{(-11,17)} \\
v_{3}^{(-5,8)} \cup v_{2}^{(-6,9)} &= -v_{2}^{(-11,17)} \\
v_{3}^{(-5,8)} \cup v_{3}^{(-6,9)} &= -3v_{3}^{(-11,17)} \\
v_{4}^{(-5,8)} \cup v_{5}^{(-6,9)} &= -2v_{1}^{(-11,17)} \\
v_{4}^{(-5,8)} \cup v_{7}^{(-6,9)} &= -2v_{2}^{(-11,17)} \\
v_{5}^{(-5,8)} \cup v_{2}^{(-6,9)} &= -2v_{1}^{(-11,17)} \\
v_{5}^{(-5,8)} \cup v_{6}^{(-6,9)} &= 2v_{3}^{(-11,17)} \\
v_{6}^{(-5,8)} \cup v_{1}^{(-6,9)} &= 2v_{2}^{(-11,17)} \\
v_{6}^{(-5,8)} \cup v_{4}^{(-6,9)} &= 2v_{3}^{(-11,17)}
\end{align*}
\vspace*{1em}

\begin{center}
  \scalebox{1.2}{$H^{-5,8} \cup H^{-10,15} \subseteq H^{-15,23}$}
\end{center}

\begin{align*}
v_{1}^{(-5,8)} \cup v_{5}^{(-10,15)} &= -v_{1}^{(-15,23)} \\
v_{2}^{(-5,8)} \cup v_{3}^{(-10,15)} &= v_{1}^{(-15,23)} \\
v_{3}^{(-5,8)} \cup v_{2}^{(-10,15)} &= -v_{1}^{(-15,23)} \\
v_{4}^{(-5,8)} \cup v_{6}^{(-10,15)} &= v_{1}^{(-15,23)} \\
v_{5}^{(-5,8)} \cup v_{4}^{(-10,15)} &= v_{1}^{(-15,23)} \\
v_{6}^{(-5,8)} \cup v_{1}^{(-10,15)} &= v_{1}^{(-15,23)}
\end{align*}
\vspace*{1em}

\begin{center}
  \scalebox{1.2}{$H^{-6,9} \cup H^{-6,9} \subseteq H^{-12,18}$}
\end{center}

\begin{align*}
v_{1}^{(-6,9)} \cup v_{3}^{(-6,9)} &= v_{1}^{(-12,18)} \\
v_{1}^{(-6,9)} \cup v_{5}^{(-6,9)} &= -v_{2}^{(-12,18)} \\
v_{1}^{(-6,9)} \cup v_{6}^{(-6,9)} &= v_{3}^{(-12,18)} \\
v_{2}^{(-6,9)} \cup v_{3}^{(-6,9)} &= v_{3}^{(-12,18)} \\
v_{2}^{(-6,9)} \cup v_{4}^{(-6,9)} &= -v_{1}^{(-12,18)} \\
v_{2}^{(-6,9)} \cup v_{7}^{(-6,9)} &= v_{4}^{(-12,18)} \\
v_{3}^{(-6,9)} \cup v_{1}^{(-6,9)} &= -v_{1}^{(-12,18)} \\
v_{3}^{(-6,9)} \cup v_{2}^{(-6,9)} &= -v_{3}^{(-12,18)} \\
v_{3}^{(-6,9)} \cup v_{4}^{(-6,9)} &= -v_{2}^{(-12,18)} \\
v_{3}^{(-6,9)} \cup v_{5}^{(-6,9)} &= -v_{5}^{(-12,18)} \\
v_{3}^{(-6,9)} \cup v_{6}^{(-6,9)} &= -v_{4}^{(-12,18)} \\
v_{3}^{(-6,9)} \cup v_{7}^{(-6,9)} &= -v_{6}^{(-12,18)} \\
v_{4}^{(-6,9)} \cup v_{2}^{(-6,9)} &= v_{1}^{(-12,18)} \\
v_{4}^{(-6,9)} \cup v_{3}^{(-6,9)} &= v_{2}^{(-12,18)} \\
v_{4}^{(-6,9)} \cup v_{7}^{(-6,9)} &= -v_{5}^{(-12,18)} \\
v_{5}^{(-6,9)} \cup v_{1}^{(-6,9)} &= v_{2}^{(-12,18)} \\
v_{5}^{(-6,9)} \cup v_{3}^{(-6,9)} &= v_{5}^{(-12,18)} \\
v_{5}^{(-6,9)} \cup v_{6}^{(-6,9)} &= -v_{6}^{(-12,18)} \\
v_{6}^{(-6,9)} \cup v_{1}^{(-6,9)} &= -v_{3}^{(-12,18)} \\
v_{6}^{(-6,9)} \cup v_{3}^{(-6,9)} &= v_{4}^{(-12,18)} \\
v_{6}^{(-6,9)} \cup v_{5}^{(-6,9)} &= v_{6}^{(-12,18)} \\
v_{7}^{(-6,9)} \cup v_{2}^{(-6,9)} &= -v_{4}^{(-12,18)} \\
v_{7}^{(-6,9)} \cup v_{3}^{(-6,9)} &= v_{6}^{(-12,18)} \\
v_{7}^{(-6,9)} \cup v_{4}^{(-6,9)} &= v_{5}^{(-12,18)}
\end{align*}
\vspace*{1em}

\begin{center}
  \scalebox{1.2}{$H^{-6,9} \cup H^{-8,12} \subseteq H^{-14,21}$}
\end{center}

\begin{align*}
v_{1}^{(-6,9)} \cup v_{3}^{(-8,12)} &= v_{1}^{(-14,21)} \\
v_{1}^{(-6,9)} \cup v_{9}^{(-8,12)} &= v_{2}^{(-14,21)} \\
v_{2}^{(-6,9)} \cup v_{3}^{(-8,12)} &= -v_{2}^{(-14,21)} \\
v_{2}^{(-6,9)} \cup v_{6}^{(-8,12)} &= v_{1}^{(-14,21)} \\
v_{3}^{(-6,9)} \cup v_{1}^{(-8,12)} &= 3v_{1}^{(-14,21)} \\
v_{3}^{(-6,9)} \cup v_{2}^{(-8,12)} &= 3v_{2}^{(-14,21)} \\
v_{3}^{(-6,9)} \cup v_{3}^{(-8,12)} &= 3v_{3}^{(-14,21)} \\
v_{4}^{(-6,9)} \cup v_{2}^{(-8,12)} &= -v_{1}^{(-14,21)} \\
v_{4}^{(-6,9)} \cup v_{8}^{(-8,12)} &= v_{3}^{(-14,21)} \\
v_{4}^{(-6,9)} \cup v_{9}^{(-8,12)} &= v_{3}^{(-14,21)} \\
v_{5}^{(-6,9)} \cup v_{2}^{(-8,12)} &= v_{3}^{(-14,21)} \\
v_{5}^{(-6,9)} \cup v_{4}^{(-8,12)} &= v_{1}^{(-14,21)} \\
v_{6}^{(-6,9)} \cup v_{1}^{(-8,12)} &= v_{2}^{(-14,21)} \\
v_{6}^{(-6,9)} \cup v_{6}^{(-8,12)} &= v_{3}^{(-14,21)} \\
v_{6}^{(-6,9)} \cup v_{7}^{(-8,12)} &= v_{3}^{(-14,21)} \\
v_{7}^{(-6,9)} \cup v_{1}^{(-8,12)} &= -v_{3}^{(-14,21)} \\
v_{7}^{(-6,9)} \cup v_{5}^{(-8,12)} &= -v_{2}^{(-14,21)}
\end{align*}
\vspace*{1em}

\clearpage

\begin{center}
  \scalebox{1.2}{$H^{-6,9} \cup H^{-9,14} \subseteq H^{-15,23}$}
\end{center}

\begin{align*}
v_{1}^{(-6,9)} \cup v_{7}^{(-9,14)} &= v_{1}^{(-15,23)} \\
v_{2}^{(-6,9)} \cup v_{6}^{(-9,14)} &= v_{1}^{(-15,23)} \\
v_{3}^{(-6,9)} \cup v_{1}^{(-9,14)} &= -v_{1}^{(-15,23)} \\
v_{4}^{(-6,9)} \cup v_{5}^{(-9,14)} &= -v_{1}^{(-15,23)} \\
v_{5}^{(-6,9)} \cup v_{4}^{(-9,14)} &= v_{1}^{(-15,23)} \\
v_{6}^{(-6,9)} \cup v_{3}^{(-9,14)} &= -v_{1}^{(-15,23)} \\
v_{7}^{(-6,9)} \cup v_{2}^{(-9,14)} &= v_{1}^{(-15,23)}
\end{align*}
\vspace*{1em}

\begin{center}
  \scalebox{1.2}{$H^{-7,11} \cup H^{-8,12} \subseteq H^{-15,23}$}
\end{center}

\begin{align*}
v_{1}^{(-7,11)} \cup v_{9}^{(-8,12)} &= v_{1}^{(-15,23)} \\
v_{2}^{(-7,11)} \cup v_{8}^{(-8,12)} &= -v_{1}^{(-15,23)} \\
v_{3}^{(-7,11)} \cup v_{3}^{(-8,12)} &= v_{1}^{(-15,23)} \\
v_{4}^{(-7,11)} \cup v_{2}^{(-8,12)} &= -v_{1}^{(-15,23)} \\
v_{5}^{(-7,11)} \cup v_{7}^{(-8,12)} &= v_{1}^{(-15,23)} \\
v_{6}^{(-7,11)} \cup v_{6}^{(-8,12)} &= -v_{1}^{(-15,23)} \\
v_{7}^{(-7,11)} \cup v_{1}^{(-8,12)} &= -v_{1}^{(-15,23)} \\
v_{8}^{(-7,11)} \cup v_{5}^{(-8,12)} &= v_{1}^{(-15,23)} \\
v_{9}^{(-7,11)} \cup v_{4}^{(-8,12)} &= -v_{1}^{(-15,23)}
\end{align*}


\end{multicols}


\end{document}